\documentclass[11pt]{article}
\usepackage{pgfplots}
\pgfplotsset{compat=1.17}
\usepgfplotslibrary{fillbetween}
\usepackage[margin =1in]{geometry}
\usepackage{amssymb,amsthm}
\usepackage{amsmath}
\numberwithin{equation}{section}

\newcommand{\R}{{\bf R}}

\newcommand{\norm}[1] {\left \| #1 \right \|}
\newcommand{\inclu}[0] {\ar@{^{(}->}}

\newcommand{\dist}{{\rm dist}}

\newcommand{\RR}{\mathbb{R}}

\newcommand{\st}{\text{subject to: }}

\newcommand{\cX}{\mathcal{X}}

\newcommand{\ceil}[1]{\left \lceil #1 \right \rceil }
\newcommand{\floor}[1]{\left \lfloor #1 \right \rfloor }

\newcommand{\argmin}{\operatornamewithlimits{argmin}}

\newcommand{\argmax}{\text{argmax}}

\newcommand{\dotp}[1]{\left\langle #1\right\rangle}

\newtheorem{theorem}{Theorem}[section]

\newtheorem{definition}[theorem]{Definition}
\newtheorem{proposition}[theorem]{Proposition}
\newtheorem{lemma}[theorem]{Lemma}
\newtheorem{corollary}[theorem]{Corollary}

\newtheorem{assumption}{Assumption}

\newcommand{\iscert}{\texttt{is\_cert}}

\newtheorem{remark}{Remark}

\theoremstyle{remark}

\usepackage{mathtools}

\global\long\def\norm#1{\|#1\|}%
\global\long\def\inner#1#2{\langle#1,#2\rangle}%
\global\long\def\argmin{\arg\min}%
\global\long\def\argmax{\arg\max}%
\global\long\def\R{\mathbb{R}}%
\global\long\def\ceil#1{\lceil#1\rceil}%
\global\long\def\R{\mathbb{R}}%
\global\long\def\ceil#1{\lceil#1\rceil}%
\global\long\def\L{\mathcal{L}}%

\global\long\def\floor#1{\lfloor#1\rfloor}%
\global\long\def\dist{\text{dist}}%

\usepackage{xcolor}
\usepackage{cite}
\usepackage{hyperref}
\usepackage{algorithm}
\usepackage{enumitem}
\usepackage{algpseudocode}
\usepackage{caption}
\usepackage{subcaption}
\usepackage{graphicx}
\usepackage{algorithm}
\usepackage{algpseudocode}
\usepackage{float}
\usepackage{babel}
\usepackage{verbatim}
\usepackage{mathtools}
\usepackage{booktabs}
\usepackage{enumitem}
\usepackage{bm}
\usepackage{algorithm, algpseudocode, multirow}
\usepackage{amsmath}
\numberwithin{equation}{section}
\usepackage{amsthm}
\usepackage{amssymb}
\usepackage{minitoc}
\usepackage{color,xcolor}

\title{Optimal Parameter-Free First-Order Methods for Convex Optimization with Unknown Growth and Smoothness}

\author{Liwei Jiang\thanks{Edwardson School of Industrial Engineering, Purdue University,
		West Lafayette, IN 47907, USA;
		\texttt{https://liwei-jiang97.github.io/}. }
	\and Ke Tang \thanks{Edwardson School of Industrial Engineering, Purdue University,
		West Lafayette, IN 47907, USA. }
	\and Zhe Zhang\thanks{Edwardson School of Industrial Engineering, Purdue University,
		West Lafayette, IN 47907, USA;
		\texttt{https://sites.google.com/view/jimmy-zhe-zhang/home}.}
}
\date{\today}

\begin{document}

\maketitle
\begin{abstract}
	We study deterministic first-order minimization of a  convex function without prior knowledge of the objective's growth,
	smoothness regime,  or associated parameters. We develop anytime, parameter-free
	bundle-level methods that adapt simultaneously to these unknown properties and
	attain best-known oracle complexities. For nonsmooth $M_0$-Lipschitz objectives satisfying
	quadratic growth with modulus $\mu$, the proposed bundle-level W-certificate
	method (BLW) achieves an oracle complexity of
	$O(M_0^2/(\mu\epsilon))$, matching the optimal rate without requiring the
	growth modulus or target accuracy as input. We then introduce an accelerated
	variant, A-BLW. Without knowing the H\"older exponent $\rho$, the smoothness
	constant $M_\rho$, the quadratic-growth modulus, or the target accuracy, A-BLW
	attains the optimal rates in the nonsmooth, weakly smooth, and smooth regimes.
	Central to both methods is an affine W-certificate, a condition based on the
	descent-slowness of an affine minorant that converts the geometry of a bundle
	model into an optimality-gap guarantee under quadratic growth. A stopping-time
	analysis further shows that the same A-BLW algorithm, without modification,
	achieves the corresponding best-known rates for general convex objectives and for
	objectives satisfying H\"older growth of order $\alpha\ge 2$. Numerical
	experiments illustrate the practical performance of the
	proposed methods.
\end{abstract}

\section{Introduction}
Consider the convex optimization problem
\begin{align}\label{eqn:opt}
	f^* = \min_{x \in X} f(x),
\end{align}
where $X$ is a closed convex set and $f: X \rightarrow \RR$ is a closed convex function. Let $X^*$ denote the set of minimizers of $f$ on $X$. We assume that $X^*$ is nonempty. We consider the classical black-box setting: an algorithm can query a first-order oracle at any point $x \in X$, and the oracle returns the function value $f(x)$ together with a subgradient $f'(x) \in \partial f(x)$. When $f$ happens to be  differentiable at $x$, the oracle returns a gradient, i.e., $f'(x) = \nabla f(x)$.

A vast literature studies first-order methods for convex optimization under regularity conditions that quantify the smoothness of the objective function. Many standard smoothness assumptions can be unified through H\"older smoothness. Formally, we say that $f$ is $(\rho,M_\rho)$-H\"older smooth on $X$, for some $\rho \in [0,1]$ and $M_\rho \ge 0$, if, for every $x,y\in X$ and every $f'(x)\in\partial f(x)$,
\begin{align*}
	0\le f(y)-f(x)-\langle f'(x),y-x\rangle
	\le \frac{M_\rho}{1+\rho}\|y-x\|^{1+\rho}.
\end{align*}
This model-error condition interpolates between the smooth and nonsmooth regimes. When $\rho=1$, it gives the usual quadratic model-error bound satisfied by functions with Lipschitz-continuous gradients. When $\rho=0$, it becomes a linear model-error bound that covers standard nonsmooth Lipschitz objectives; in particular, if all subgradients have norm at most $L$, then the condition holds with $M_0=2L$. The intermediate regime $\rho \in (0,1)$ is often called weak smoothness. Under these assumptions, one can design optimal first-order methods when the relevant problem parameters are known~\cite{complexity}. Moreover, there exist ``universal'' methods that adapt to unknown smoothness parameters while still achieving the optimal rate; see, for example,~\cite{lan2015bundle, nesterov2015universal,li2025simple}.

In practice, however, the observed convergence rate is often faster than the worst-case rates guaranteed by the aforementioned results, a phenomenon also noted in their numerical experiments~\cite{li2025simple}. One explanation is that practical objectives often satisfy a growth condition, which is also known as the  error bound and is closely related to the Kurdyka--{\L}ojasiewicz property~\cite{bolte2007lojasiewicz}. For illustrative purposes, we begin with the important special case of quadratic growth and then discuss the more general H\"older growth condition. We say that $f$ satisfies quadratic growth with modulus $\mu$ if
\begin{align*}
	f(x) - f^* \ge \frac{\mu}{2} \dist^2(x,\cX^*)
	\quad \text{for all $x \in \{z \in X \colon f(z) \le f(\bar x_0)\}$},
\end{align*}
where $\bar x_0$ denotes the initial point. Despite its appeal, there remains a long-standing gap between parameter-aware and parameter-free methods for exploiting this property, even in the nonsmooth case $\rho=0$. On the one hand, when $M_0$ and the true value of $\mu$ are known to the user, the projected  subgradient method with a properly chosen stepsize achieves the optimal complexity $
	O\left(\frac{M_0^2}{\mu \epsilon}\right)$~\cite{bubeck2015convex}. On the other hand, existing parameter-free methods fail to attain this optimal complexity when the growth modulus $\mu$ is unknown a priori. Thus, a gap remains between methods that require prior knowledge of $\mu$ and those that do not. Nesterov~\cite{nesterov2025universal} recently highlighted this open question once again as the ``most interesting direction.''

In this work, we answer this open question. We first develop an optimal, anytime, parameter-free first-order method for nonsmooth objectives with quadratic growth. Our method can be viewed as a variant of the bundle-level method~\cite{lemarechal1995new}, equipped with a new termination criterion that exploits the quadratic growth condition. This criterion enables the method to achieve the optimal complexity $
	O\left(\frac{M_0^2}{\mu \epsilon}\right)$
without requiring any knowledge of the problem parameters. Building on this result, we design an anytime universal method for functions satisfying both quadratic growth and H\"older smoothness conditions. The method requires no prior knowledge of either the problem regime or the relevant problem parameters, yet achieves optimal rates across all regimes. Unlike  existing universal methods in the convex setting $(\mu=0)$~\cite{nesterov2015universal, li2025simple}, our parameter-free method is genuinely anytime: it does not require the target accuracy as an input.

Finally, we extend our universal guarantees  beyond the quadratic-growth regime to two broader settings: general convex problems with no growth condition, and problems satisfying an $\alpha$-H\"older growth condition $(\alpha \ge 2)$ of the form
\begin{align*}
	f(x)-f^* \ge \mu_\alpha\,\operatorname{dist}^\alpha(x,X^*),  \quad \text{for all $x \in \{z \in X \colon f(z) \le f(\bar x_0)\}$}.
\end{align*}
The resulting method is universal in a stronger sense: without knowing $\mu_\alpha$, $\alpha$, the H\"older smoothness regime or parameters, or the target accuracy, it matches the best-known first-order complexity bound for the corresponding problem class.  Table~\ref{tab:section42-rates} shows the oracle complexity of our proposed method. In the row $\rho\in(0,1)$, the notation $O_\rho(\cdot)$ allows constants depending on the fixed smoothness exponent $\rho$. In the $\rho=1$, $\alpha>2$ entry, $O_\alpha(\cdot)$ allows constants depending only on the fixed growth exponent $\alpha$.

\begin{table}[H]
	\centering
	\small
	\newcommand{\ratecorner}{%
		\begin{tikzpicture}[baseline=(current bounding box.center)]
			\useasboundingbox (0,0) rectangle (3.7cm,1.25cm);
			\draw[line width=0.35pt] (0,1.25) -- (3.7,0);
			\node[anchor=north east,inner sep=1.5pt] at (3.55,1.14) {Growth};
			\node[anchor=south west,inner sep=1.5pt] at (0.12,0.10) {Smoothness};
		\end{tikzpicture}%
	}
	\renewcommand{\arraystretch}{2.1}
	\resizebox{\textwidth}{!}{%
		\begin{tabular}{@{}c@{}|ccc@{}}
			\hline
			\ratecorner
			 & General convex
			 & $\alpha=2$
			 & $\alpha>2$     \\
			\hline
			$\rho=0$
			 &
			$\displaystyle O\!\left(\frac{M_0^2\dist^2(\bar x_0,X^*)}{\epsilon^2}\right)$
			 &
			$\displaystyle O\!\left(\frac{M_0^2}{\mu_2\epsilon}\right)$
			 &
			$\displaystyle O\!\left(\frac{M_0^2}{\mu_\alpha^{2/\alpha}\epsilon^{2(\alpha-1)/\alpha}}\right)$
			\\
			$\rho\in(0,1)$
			 &
			$\displaystyle O_\rho\!\left(
				\left(\frac{M_\rho^2\dist^{2(1+\rho)}(\bar x_0,X^*)}{\epsilon^2}\right)^{\frac{1}{1+3\rho}}
				\right)$
			 &
			$\displaystyle O_\rho\!\left(
				\left(\frac{M_\rho^2}{\mu_2^{1+\rho}\epsilon^{1-\rho}}\right)^{\frac{1}{1+3\rho}}
				\right)$
			 &
			$\displaystyle O_\rho\!\left(
				\left(\frac{M_\rho^2}{\mu_\alpha^{2(1+\rho)/\alpha}\epsilon^{2(\alpha-1-\rho)/\alpha}}\right)^{\frac{1}{1+3\rho}}
				\right)$
			\\
			$\rho=1$
			 &
			$\displaystyle O\!\left(\frac{\sqrt{M_1}\,\dist(\bar x_0,X^*)}{\sqrt{\epsilon}}\right)$
			 &
			$\displaystyle O\!\left(\sqrt{\frac{M_1}{\mu_2}}\log\frac{1}{\epsilon}\right)$
			 &
			$\displaystyle O_\alpha\!\left(
				\frac{\sqrt{M_1}}{\mu_\alpha^{1/\alpha}\epsilon^{(\alpha-2)/(2\alpha)}}
				\right)$
			\\
			\hline
		\end{tabular}%
	}
	\caption{Oracle complexities for finding a point $\bar x$ with $f(\bar x)-f^*\le\epsilon$.}
	\label{tab:section42-rates}
\end{table}

The rest of the paper is organized as follows. We close this section with related work and notation. Section~\ref{sec:affine_certificate} introduces the affine W-certificate and proves how it  yields an optimality-gap bound under quadratic growth. Section~\ref{sec:bundle level} builds on this certificate to develop BLW, a parameter-free method for nonsmooth convex functions with quadratic growth, together with its optimal anytime complexity guarantee. Section~\ref{sec:ablw-holder} develops the accelerated A-BLCI subroutine and the A-BLW method and establishes the uniform optimality of A-BLW.

\subsection{Related work}\label{sec:related_work}
We discuss the results in deterministic optimization that are most closely related to our paper, organized under two themes: parameter-free optimization and algorithms that exploit growth conditions. These two themes have developed largely in parallel. The former focuses on removing the need to know smoothness constants, Lipschitz constants, distances to the solution set, or other tuning parameters. The latter focuses on exploiting error bounds, sharpness, quadratic growth, or more general H\"olderian growth conditions to obtain faster rates than those guaranteed in the worst-case convex setting. Our work lies at the intersection of these two directions: we aim to exploit growth without knowing the growth parameters, the smoothness regime, the smoothness constants, or the target accuracy.

\paragraph{Parameter-free optimization.}
A first line of related work concerns universal methods for convex optimization under H\"older smoothness. Lan's accelerated bundle-level methods provide an important route to uniform optimality, achieving optimal complexity bounds for convex optimization without requiring the user to specify the smoothness parameters~\cite{lan2015bundle}.  Nesterov's fast gradient method is another parameter-free strategy that also recovers the optimal rates for smooth, weakly smooth, and nonsmooth convex optimization~\cite{nesterov2015universal}. More recently,  line-search-free methods have been proposed~\cite{malitsky2018golden}. In particular, Li and Lan further designed a line-search-free method that optimally adapts to H\"older smoothness in the convex setting~\cite{li2025simple}. Related parameter-free methods have also been developed for a variety of other settings~\cite{zhou2026adabb, suh2025adaptive, lan2024auto, lan2024projected,deng2026uniformly,yagishita2025simple,ye2025simple,giang2026auto,ji2026stochastic,wu2026universal, li2024problem, borodich2025nesterov,guigues2026universal,rodomanov2024universal,ou2025linesearch,marumo2024parameter}.

These results are most closely related to the smoothness-adaptation component of our work. However, the above parameter-free methods primarily address the standard convex setting and are not designed to exploit an unknown growth modulus. There are a few exceptions: the works~\cite{malitsky2020adaptive, borodich2025nesterov, lan2024optimalparameterfreegradientminimization} develop parameter-free methods that can exploit the growth condition of the problem, but their analyses are mainly restricted to the smooth setting. Restart schemes provide another important approach to adapting to growth conditions~\cite{lin2015adaptive,fercoq2017restart,liu2017adaptive, roulet2020sharpness,renegar2022restart, sujanani2025efficient, ito2021nearly}. However, these schemes typically incur a nonconstant multiplicative overhead in the oracle complexity, which makes them suboptimal. Notably, recent work by Wu and Grimmer~\cite{wu2026restart} shows that this overhead can be reduced to a log-logarithmic factor and can be parallelized. To the best of our knowledge, even under quadratic growth, no prior work has developed a parameter-free method that optimally adapts to the growth condition in the nonsmooth or weakly smooth setting.

\paragraph{Algorithms exploiting growth conditions.}
Growth conditions have a long history in optimization and variational analysis. They hold for broad, generic classes of functions~\cite{burke1993weak, drusvyatskiy2016generic} and are closely connected to metric subregularity and the behavior of optimization algorithms~\cite{burke1993weak,drusvyatskiy2018error,bolte2017error}. Beyond the restart schemes discussed above, exploiting growth conditions together with additional problem structure often leads to much faster convergence guarantees. For example, in smooth convex optimization, quadratic growth and related error-bound conditions explain why first-order methods can converge linearly even when the objective is not strongly convex~\cite{karimi2016linear}. More recently, it has been shown that gradient descent with alternating constant and Polyak stepsizes enjoys nearly linear convergence for smooth functions satisfying quartic growth~\cite{davis2025gradient}. In nonsmooth optimization, sharp growth can lead to exponentially faster convergence in suitable structured settings~\cite{davis2018sharp,davis2024stochastic, charisopoulos2024superlinearly}. Similar acceleration phenomena have also been established under quadratic growth for several important structured nonsmooth problem classes~\cite{han2023survey, davis2025local, kong2025lipschitz,li2025subgradient, zhang2025linearly, lin2026optimal}. However, these results rely on additional structural assumptions beyond growth alone, and therefore their acceleration guarantees do not apply to the general black-box model considered in this paper.

\subsection{Notation}\label{sec:notation}
In this paper, we let $\dotp{\cdot,\cdot}$ denote the Euclidean inner product and $\|\cdot\|$ denote the corresponding Euclidean norm.  For $a,b \in \RR$, we write  $a \vee b := \max\{a,b\}$. For $a \in \RR$, $\ceil{a}$ denotes the smallest integer greater than or equal to $a$ and $\floor{a}$ denotes the largest integer less than or equal to $a$. For any convex function $f$, we denote its convex subdifferential at a point $x$ by $\partial f(x)$ and let $f'(x)$ denote an arbitrary subgradient in $\partial f(x)$. We use $\ell_f(x; y)$ to denote the linear model of $f$ constructed at point $y$ and evaluated at point $x$:
\[
	\ell_f(x; y) := f(y) + \langle f'(y), x - y \rangle .
\]
We let $f^* = \inf_{x \in X} f(x)$ denote the optimal value of $f$ over $X$. We also let $X^* = \arg\min_{x \in X} f(x)$ denote the set of minimizers of $f$ on $X$.

\section{Affine W-certificate}\label{sec:affine_certificate}
In this section and the next, we develop optimal parameter-free algorithms for minimizing nonsmooth convex functions that satisfy a quadratic growth condition. We begin by defining two quantities that are useful throughout this paper. We say that a function $g$ is an affine minorant of $f$ if $g$ is an affine function and $g(x) \le f(x)$ for all $x\in X$.
\begin{definition}
	For any affine minorant $h$ of $f$, any point $\bar x \in X$, and  any descent quantity $\Delta >0$, define
	\begin{equation}\label{eq:d-delta}
		d(h, \bar x, \Delta) :=
		\inf\left\{\|\bar x-x\|: x\in X,\ h(x) \le f(\bar x)-\Delta\right\},
	\end{equation}
	with the convention that the infimum of an empty set is $+\infty$.
	We also define the \emph{descent-slowness}
	$$s(h,\bar x,\Delta) := \frac{d(h, \bar x, \Delta)}{\Delta}.$$
\end{definition}
Geometrically, $d(h,\bar x, \Delta)$ denotes the away-radius needed for the affine function to achieve descent $\Delta$, and the descent-slowness $s(h,\bar x, \Delta)$ is the average distance required for one unit of descent. If the affine minorant never reaches $f(\bar x)-\Delta$ on $X$, this radius is $+\infty$.

We next establish a few basic properties. The first lemma shows a monotonicity property of the descent-slowness.

\begin{lemma}[Monotonicity of descent-slowness]\label{lem:monotonicity}
	For any affine minorant $h$ of $f$, descent gaps $\Delta_2 \ge \Delta_1 >0$, and $\bar x \in X$, we have
	\begin{equation}\label{eq:monotonicity}
		s(h, \bar x, \Delta_1) \le s(h, \bar x, \Delta_2).
	\end{equation}
\end{lemma}

\begin{proof}
	By the minorant property, $h(\bar x)\le f(\bar x)$. Let
	\[
		d_i := d(h,\bar x,\Delta_i), \qquad i=1,2.
	\]
	If $d_2=+\infty$, then the claim is immediate. Suppose $d_2<+\infty$. Then the set $C_2:=\{x\in X:h(x)\le f(\bar x)-\Delta_2\}$ is nonempty and closed, so the Euclidean projection of $\bar x$ onto $C_2$ exists. Choose $\hat x\in C_2$ such that $\|\bar x-\hat x\|=d_2$. Define
	\[
		\theta := \frac{\Delta_1}{\Delta_2}\in (0,1],
		\qquad
		\tilde x := (1-\theta)\bar x+\theta \hat x.
	\]
	By convexity of $X$, we have $\tilde x\in X$. Since $h$ is affine and $h(\hat x)\le f(\bar x)-\Delta_2$,
	\[
		\begin{aligned}
			h(\tilde x)
			 & = (1-\theta)h(\bar x)+\theta h(\hat x)                       \\
			 & \le (1-\theta)f(\bar x)+\theta\bigl(f(\bar x)-\Delta_2\bigr)
			= f(\bar x)-\Delta_1.
		\end{aligned}
	\]
	Thus $\tilde x$ is feasible in the definition of $d_1$, and hence
	\[
		d_1 \le \|\bar x-\tilde x\|
		= \theta\|\bar x-\hat x\|
		= \frac{\Delta_1}{\Delta_2}d_2.
	\]
	Dividing by $\Delta_1$ gives $s(h,\bar x,\Delta_1)\le s(h,\bar x,\Delta_2)$.
\end{proof}

For a point $x\in X$ and a scalar $\mu>0$, we say that $x$ has pointwise quadratic growth with modulus $\mu$ if
\[
	f(x)-f^* \ge \frac{\mu}{2}\dist^2(x, X^*).
\]

The next is the key result in this section and shows how descent-slowness can provide an upper bound on the optimality gap.
\begin{proposition}[Gap bound from descent-slowness]\label{prop:gap-bound}
	Suppose that $h(\cdot)$ is an affine minorant of $f$. For any $\bar x \in X$ having pointwise quadratic growth with modulus $\mu^* > 0$ and  any $\Delta >0$, we have
	\begin{equation}\label{eq:gap-bound}
		f(\bar x)-f^* \le \Delta \vee \frac{2}{\mu^* s^2(h,\bar x,\Delta)}.
	\end{equation}
	Here we use the extended-real convention $1/0=+\infty$ and $1/(+\infty)=0$.
\end{proposition}

\begin{proof}
	Let $\Delta^*:=f(\bar x)-f^*$. If $\Delta^*\le \Delta$, the claim is immediate. If $s(h,\bar x,\Delta)=0$, the second term in \eqref{eq:gap-bound} is $+\infty$, and the claim is also immediate. Hence assume $\Delta^*>\Delta$ and $s(h,\bar x,\Delta)>0$.

	Choose $x^*\in X^*$ such that $\|\bar x-x^*\|=\dist(\bar x,\cX^*)$. The pointwise quadratic growth assumption implies that
	\[
		\frac{\mu^*}{2}\|\bar x-x^*\|^2 \le \Delta^*.
	\]
	Because $h$ is a minorant of $f$, we have
	\[
		h(x^*)\le f(x^*)=f^*=f(\bar x)-\Delta^*.
	\]
	Thus $x^*$ is feasible in the definition of $d(h,\bar x,\Delta^*)$, and hence $
		d(h,\bar x,\Delta^*)\le \|\bar x-x^*\|.$
	Since $\Delta^*>\Delta$, Lemma~\ref{lem:monotonicity} gives
	\[
		s(h,\bar x,\Delta)\le s(h,\bar x,\Delta^*).
	\]
	Moreover, $d(h,\bar x,\Delta^*)<+\infty$, so $0<s(h,\bar x,\Delta)\le s(h,\bar x,\Delta^*)<+\infty$. Therefore
	\[
		\Delta^*
		= \frac{d(h,\bar x,\Delta^*)}{s(h,\bar x,\Delta^*)}
		\le \frac{\|\bar x-x^*\|}{s(h,\bar x,\Delta)}.
	\]
	Combining this inequality with quadratic growth yields
	\[
		\frac{\mu^*}{2}\|\bar x-x^*\|^2
		\le \Delta^*
		\le \frac{\|\bar x-x^*\|}{s(h,\bar x,\Delta)}.
	\]
	Since $\Delta^*>0$, we have $\|\bar x-x^*\|>0$, and hence $
		\|\bar x-x^*\|\le \frac{2}{\mu^*s(h,\bar x,\Delta)}.$
	Substituting this bound into the upper bound on $\Delta^*$ gives
	\[
		\Delta^*
		\le \frac{2}{\mu^*s^2(h,\bar x,\Delta)}.
	\]
	This proves \eqref{eq:gap-bound}.
\end{proof}
Proposition~\ref{prop:gap-bound} suggests a natural way to certify near-optimality. For a candidate gap $\Delta$, it suffices to find an affine minorant whose descent-slowness is large enough. If the exact quadratic-growth
modulus $\mu^*$ were known, the sufficient threshold would be
\[
	s(h,\bar x,\Delta)^2 \ge \frac{2}{\mu^*\Delta}.
\]
In the parameter-free setting, however, $\mu^*$ is not available. We therefore
introduce a family of certificates indexed by a trial modulus $\mu$.

\begin{definition}[W-certificate]\label{def:valid-compressed-w}
	We say that an affine function $h$ is a $(\mu,\Delta)$  W-certificate for $f$ at $\bar x \in X$ if $h$ minorizes $f$ and
	\begin{equation}\label{eq:valid-compressed-w}
		s(h,\bar x, \Delta)^2 \ge \frac{2}{\mu\Delta}.
	\end{equation}
\end{definition}
The next corollary converts a W-certificate into an explicit optimality-gap bound. The statement separates the algorithmic parameter $\mu$ used in the
certificate from the true quadratic-growth modulus $\mu^*$ available at
$\bar x$.
\begin{corollary}\label{cor:valid-gap}
	Suppose that $\bar x \in X$ has pointwise quadratic growth with modulus $\mu^* > 0$. For  any $\Delta >0$,  if there exists  a $(\mu,\Delta)$ W-certificate for $\bar x$, then
	\begin{equation}\label{eq:valid-gap}
		f(\bar x)-f^* \le \Delta \vee \frac{\mu\Delta}{\mu^*}.
	\end{equation}
	In particular, if $\mu\le \mu^*$, then $f(\bar x)-f^*\le \Delta$.
\end{corollary}

Several comments are useful for interpreting this definition. The quantity
$\mu^*$ is a property of the objective and appears only in the analysis; it is
not required by the certificate construction. In this section, $\mu^*$ should be
understood pointwise at the evaluation center $\bar x$. Later, when we analyze
the algorithm, the quadratic-growth assumption will ensure that the same
modulus applies to every evaluation center generated by the method.

By contrast, $\mu$ is an algorithmic trial parameter. It sets the scale at which
the affine minorant is tested. A larger value of $\mu$ makes the condition
\eqref{eq:valid-compressed-w} easier to satisfy, but the resulting bound
\eqref{eq:valid-gap} becomes weaker by the factor $\mu/\mu^*$.  Thus
$\mu$ should be viewed not as an assumed problem parameter, but as the current estimate of the
modulus scale used by the algorithm. Moreover, progress toward optimality is governed not by either parameter individually, but by the product $\mu \Delta$. This observation serves as a guiding principle for our algorithmic design.

For our algorithms, it is useful to return a Boolean certificate flag  $\iscert{}$. We say that the pair $(\bar x,\mu,\Delta)$ is \emph{certified} if the gap bound~\eqref{eq:valid-gap} has been verified; equivalently, $\iscert{}=\texttt{True}$ means that the current call has verified this bound. A W-certificate is one way to verify it. Another useful case is when an affine minorant $h$ satisfies
\[
	h(x)\ge f(\bar x)-\Delta
	\qquad \text{for all } x\in X,
\]
because then $f^*\ge f(\bar x)-\Delta$ and hence $f(\bar x)-f^*\le \Delta$, which is stronger than~\eqref{eq:valid-gap}.

\section{Parameter-free method for nonsmooth convex functions with quadratic growth}\label{sec:bundle level}
Equipped with the affine W-certificate, we develop a parameter-free method for nonsmooth problems satisfying quadratic growth.  Section~\ref{subsec:blw_sub} presents  the bundle-level certify-or-improve (BLCI) algorithm, which serves as the core building block of the bundle-level W-certificate (BLW) method described in Section~\ref{subsec:outer_blw_nonsmooth}. We formally state the regularity conditions that will be imposed throughout this section.

The first assumption is the model-error form of the usual nonsmooth Lipschitz condition. It is implied by the usual Lipschitz condition on the objective value.
\begin{assumption}[Lipschitz continuity]\label{ass:model-error}
	There exists $M>0$ such that
	\begin{equation}\label{eq:model-error}
		0\le f(x)-\ell_f(x;y)\le M\|x-y\|,\qquad x,y\in X.
	\end{equation}
\end{assumption}
The second assumption is the quadratic growth condition.  Since it is assumed only on the initial sublevel set, the feasible set $X$ need not be bounded.
\begin{assumption}[Quadratic growth]\label{ass:qg}
	There exists $\mu^* > 0$ such that
	\begin{equation}\label{eq:qg}
		f(x) - f^* \ge \frac{\mu^*}{2} \dist^2(x, \cX^*) \quad  \text{for all $x \in  \{z \in X \colon f(z) \le f(\bar x_0)\}$},
	\end{equation}
	where $\bar x_0$ is the initial point.
\end{assumption}

\subsection{Bundle-level certify-or-improve subroutine}\label{subsec:blw_sub}
The BLCI method is shown in Algorithm~\ref{alg:blw}. It can be viewed as a variant of the classical bundle-level method~\cite{lemarechal1995new,lan2020first}. For a fixed evaluation center $\bar x$ and candidate parameters $(\mu,\Delta)$, BLCI starts from an affine minorant $a^{(0)}$ of $f$ and either certifies that the pair $(\bar x,\mu,\Delta)$ satisfies~\eqref{eq:valid-gap} or returns a point with substantial objective decrease.  The memory length $m$ is user-specified but should not be viewed as a tuning parameter. Indeed, all results established in this work hold for every $m \ge 1$, and the complexity bounds are independent of $m$. We therefore do not prescribe a particular value of $m$ in the statements of our results. The choice of $m$ affects only empirical performance, as demonstrated in Section~\ref{sec:numerics}.

\begin{algorithm}[H]
	\small
	\caption{The Bundle-Level Certify-or-Improve Method $\mathrm{BLCI}(\bar x,\mu,\Delta,m,a^{(0)})$}\label{alg:blw}
	\begin{algorithmic}[1]
		\Require Evaluation center $\bar x$, QG modulus estimate $\mu$, gap estimate $\Delta$, memory length $m$, and affine minorant $a^{(0)}$.
		\Ensure An affine minorant $a^+$, a point $x^+$ that is either $\bar x$ or a potentially improved solution, and a certificate flag $\iscert{}$.
		\State\label{line:blw-initial-cert} Set $\ell=f(\bar x)-\Delta$. If $\{x\in X:a^{(0)}(x)\le \ell\}=\emptyset$, return $(a^{(0)},\bar x,\texttt{True})$.
		\State\label{line:blw-initial-project}\begin{minipage}[t]{0.88\linewidth}
			Set
			\begin{equation}\label{eq:blw-x0}
				x_0\in \argmin_{x\in X}\left\{\|\bar x-x\|^2:a^{(0)}(x)\le \ell\right\}.
			\end{equation}
		\end{minipage}
		\State\label{line:blw-initial-query} Query the oracle at $x_0$ and form the cut $\ell_f(\cdot;x_0)$.
		\For{$t=1,2,3,\ldots$}\label{line:blw-for}
		\State\label{line:blw-maxmodel}\begin{minipage}[t]{0.84\linewidth}
			Form the max-model
			\[
				\psi_t(x):=a^{(t-1)}(x)\vee \max_{[t-m]_+\le i\le t-1}\ell_f(x;x_i).
			\]
		\end{minipage}
		\State\label{line:blw-loop-cert}\begin{minipage}[t]{0.84\linewidth}
			Set $c_t:=\inf_{x\in X}\psi_t(x)$. If $c_t\ge \ell$, set $a_{c_t}(\cdot)\equiv c_t$ and return $(a_{c_t},\bar x,\texttt{True})$.
		\end{minipage}
		\State\label{line:blw-project}\begin{minipage}[t]{0.84\linewidth}
			Solve
			\begin{equation}\label{eq:blw-projection-subproblem}
				\begin{aligned}
					x_t\in \argmin_{x\in X}\quad & \|\bar x-x\|^2                                      \\
					\mathrm{s.t.}\quad           & a^{(t-1)}(x)\le \ell,                               \\
					                             & \ell_f(x;x_i)\le \ell,  \quad [t-m]_+ \le i\le t-1.
				\end{aligned}
			\end{equation}
		\end{minipage}
		\State\label{line:blw-aggregate}\begin{minipage}[t]{0.84\linewidth}
			Extract optimal dual multipliers from~\eqref{eq:blw-projection-subproblem}, and normalize them to obtain $\lambda_a\in\R_+$ and $\lambda_i\in\R_+$ with $\lambda_a+\sum_i\lambda_i=1$. Set
			\[
				a^{(t)}(\cdot):=\lambda_a a^{(t-1)}(\cdot)+\sum_{i= [t-m]_+ }^{t-1}\lambda_i\ell_f(\cdot;x_i).
			\]
		\end{minipage}
		\State\label{line:blw-valid} If $\|\bar x-x_t\|^2/\Delta^2\ge 2/(\mu\Delta)$, return $(a^{(t)},\bar x,\texttt{True})$.
		\State\label{line:blw-improve} Query the oracle at $x_t$.  If $f(x_t)<f(\bar x)-4\Delta/5$, return $(a^{(t)},x_t,\texttt{False})$.
		\EndFor
	\end{algorithmic}
\end{algorithm}

In Line~\ref{line:blw-loop-cert}, since $c_t=\inf_{x\in X}\psi_t(x)$ and $\psi_t\le f$, we have $c_t\le f(x)$ for all $x\in X$, and therefore $a_{c_t}$ is an affine minorant. When $c_t\ge \ell=f(\bar x)-\Delta$, we also have $f^*\ge c_t\ge \ell$, so the gap bound~\eqref{eq:valid-gap} holds. Thus, the return with $\iscert{}=\texttt{True}$ is justified directly, and no descent-slowness value is needed in this branch.

The multiplier step in Line~\ref{line:blw-aggregate} is justified by the test in Line~\ref{line:blw-loop-cert}. The proof below verifies the relative Slater condition, the existence of optimal multipliers, and the positivity needed to normalize them.

The next lemma records the projection geometry behind the affine aggregation step.

\begin{lemma}[Projection three-point inequality]\label{lem:projection-three-point}
	Consider iterates $\{x_t\}$ and aggregate models $a^{(t)}$ generated by Algorithm~\ref{alg:blw}. For every aggregate model generated before termination,
	\begin{equation}\label{eq:three-point}
		\|\bar x-x_t\|^2 + \|x_t-\hat x\|^2
		\le \|\bar x-\hat x\|^2,
		\qquad \forall \hat x\in\{x\in X:a^{(t)}(x)\le \ell\}.
	\end{equation}
	Consequently, whenever $x_{t+1}$ is generated,
	\begin{equation}\label{eq:three-point-step}
		\|\bar x-x_t\|^2 + \|x_t-x_{t+1}\|^2
		\le \|\bar x-x_{t+1}\|^2.
	\end{equation}
\end{lemma}

\begin{proof}
	Let $\mathcal H_t$ consist of $a^{(t-1)}$ and the cuts
	$\ell_f(\cdot;x_i)$ appearing in~\eqref{eq:blw-projection-subproblem}.
	Since $c_t<\ell$, there exists $z\in X$ such that
	$h(z)<\ell$ for every $h\in\mathcal H_t$. Thus, the Slater condition holds, and the optimal
	multipliers $\eta_h\ge0$ exist and satisfy
	\begin{equation}\label{eq:blw-kkt}
		0\in 2(x_t-\bar x)+N_X(x_t)
		+\sum_{h\in\mathcal H_t}\eta_h\nabla h,
		\qquad
		\eta_h\bigl(h(x_t)-\ell\bigr)=0.
	\end{equation}
	Here $N_X(x_t)$ denotes the normal cone to $X$ at $x_t$.

	We next show inductively that $\bar x$ violates at least one constraint in
	each projection subproblem. At $t=1$, this is immediate if
	$a^{(0)}(\bar x)>\ell$. Otherwise, $\bar x$ is feasible in the definition
	of $x_0$, so $x_0=\bar x$ and
	$\ell_f(\bar x;x_0)=f(\bar x)>\ell$. Hence
	$\psi_1(\bar x)>\ell$. For later iterations, the induction will follow from
	$a^{(t-1)}(\bar x)>\ell$, which we establish below. In particular,
	$x_t\ne\bar x$.

	Set $\Lambda:=\sum_{h\in\mathcal H_t}\eta_h$. If $\Lambda=0$, then
	\eqref{eq:blw-kkt} is the optimality condition for projecting $\bar x$ onto
	$X$. Since $\bar x\in X$, it would imply $x_t=\bar x$, a contradiction.
	Thus $\Lambda>0$. Normalize the multipliers by
	$\lambda_h:=\eta_h/\Lambda$ and let $a^{(t)}$ be their convex combination,
	as in Line~\ref{line:blw-aggregate}. Complementarity gives
	$a^{(t)}(x_t)=\ell$, while~\eqref{eq:blw-kkt} becomes
	\[
		0\in2(x_t-\bar x)+N_X(x_t)+\Lambda\nabla a^{(t)}.
	\]
	These are precisely the KKT conditions for projecting $\bar x$ onto the
	closed convex set
	\[
		C_t := \{x\in X:a^{(t)}(x)\le \ell\}.
	\]
	Consequently, $x_t$ is that projection. Moreover,
	$a^{(t)}(\bar x)>\ell$: otherwise $\bar x\in C_t$, contradicting
	$x_t\ne\bar x$. This completes the induction used above and also justifies
	the normalized aggregate in every generated iteration.

	Therefore
	\[
		\left\langle x_t-\bar x,\hat x-x_t\right\rangle\ge 0, \qquad \forall \hat x\in C_t.
	\]
	Expanding the square gives
	\[
		\begin{aligned}
			\|\bar x-\hat x\|^2
			 & = \|\bar x-x_t\|^2 + \|x_t-\hat x\|^2
			+2\left\langle x_t-\bar x,\hat x-x_t\right\rangle \\
			 & \ge \|\bar x-x_t\|^2 + \|x_t-\hat x\|^2,
		\end{aligned}
	\]
	which proves \eqref{eq:three-point}. Since $x_{t+1}$ is feasible for $a^{(t)}(x)\le \ell$ whenever it is generated, \eqref{eq:three-point-step} follows.
\end{proof}
\begin{remark}
	Using dual variables in Line~\ref{line:blw-aggregate} to reduce storage requirements is not a new idea; related techniques have appeared in~\cite{cox2014dual}. From an implementation standpoint, Line~\ref{line:blw-aggregate}
	requires the projection subproblem to be solved together with one set of
	optimal KKT multipliers. This is a computational assumption on the
	representation of \(X\) and on the subproblem solver, rather than a
	theoretical requirement that \(X\) belong to a particular class of
	simple sets. The multiplier computation is especially tractable when
	projection onto \(X\) is efficient, as is the case for
	\(\mathbb{R}^d\), boxes, simplices, and Euclidean balls.

	In particular, when \(X=\mathbb{R}^d\), the iterates generated by Algorithm~\ref{alg:blw} satisfy the standard linear-span condition for first-order methods. Moreover, when \(m=1\), each projection subproblem reduces to projecting $\bar x$
	onto the intersection of the aggregate halfspace and a single oracle-cut halfspace, and the associated optimal multiplier vector admits a closed-form expression.
\end{remark}
Putting the iterations together, the next proposition gives a rigorous guarantee for BLCI calls, including iteration complexity and sufficient conditions for certification. Notice that the descent-slowness progress bound also depends on the initial affine minorant. Hence, a better initial affine minorant can warm-start the algorithm and save oracle evaluations, which will be important for obtaining the desired complexity bound.

\begin{proposition}\label{prop:blw-basic}
	Suppose that Assumptions~\ref{ass:model-error} and~\ref{ass:qg} hold, and suppose that the evaluation center $\bar x$ lies in the initial level set from Assumption~\ref{ass:qg}. Consider one call to Algorithm~\ref{alg:blw} with parameters  $\bar x$, $\mu$, $\Delta$, and initial affine minorant $a^{(0)}$. Let $N_{\mathrm{ter}}$ denote the last iteration. Then the following properties hold for the iterates generated by this BLCI call.
	\begin{enumerate}[label=\arabic*.,ref=\arabic*,leftmargin=2em]
		\item\label{item:blw-slowness-progress} For every queried point $x_t$ with $t<N_{\mathrm{ter}}$,
		      \begin{equation}\label{eq:slowness-progress}
			      s(a^{(t)},\bar x,\Delta)^2
			      \ge s(a^{(0)},\bar x,\Delta)^2 + \frac{t}{25M^2}.
		      \end{equation}
		      Consequently, if Algorithm~\ref{alg:blw} terminates at Line~\ref{line:blw-improve} with a significantly improved solution, the associated affine minorant satisfies
		      \begin{equation}\label{eq:slowness-progress-final}
			      s(a^{(N_{\mathrm{ter}})},\bar x,\Delta)^2
			      \ge s(a^{(0)},\bar x,\Delta)^2 + \frac{N_{\mathrm{ter}}-1}{25M^2}.
		      \end{equation}

		\item\label{item:blw-termination} The BLCI algorithm terminates within $50M^2/(\mu\Delta)+2$ oracle evaluations.

		\item\label{item:blw-cert-flag-gap} When the BLCI algorithm returns with $\iscert{}=\texttt{True}$ for the pair $(\bar x,\mu,\Delta)$, we have
		      \begin{align*}
			      f(\bar x)-f^* \le \Delta \vee \frac{\mu\Delta}{\mu^*}.
		      \end{align*}

		\item\label{item:blw-small-gap-cert} If $f(\bar x)-f^*\le 4\Delta/5$, the BLCI algorithm returns with $\iscert{}=\texttt{True}$, and hence certifies the pair $(\bar x,\mu,\Delta)$.
	\end{enumerate}
\end{proposition}

\begin{proof}
	We first record two elementary facts that will be used repeatedly. Since $a^{(0)}$ is an affine minorant of $f$ and every new aggregate model is a convex combination of previous aggregate models and oracle linearizations, each $a^{(t)}$ remains an affine minorant of $f$. Moreover, by the aggregate construction used in Algorithm~\ref{alg:blw} and Lemma~\ref{lem:projection-three-point}, $x_t$ is the Euclidean projection of $\bar x$ onto $\{x\in X:a^{(t)}(x)\le \ell\}$. Hence, whenever this set is nonempty and $x_t$ exists,
	\[
		d(a^{(t)},\bar x,\Delta)=\|\bar x-x_t\|,
		\qquad
		s(a^{(t)},\bar x,\Delta)=\frac{\|\bar x-x_t\|}{\Delta}.
	\]
	We now prove the one-step progress estimate. Fix an iteration $j$ before termination. Since the algorithm has not terminated at Line~\ref{line:blw-improve}  of Algorithm~\ref{alg:blw} in  this iteration,
	\[
		f(x_j)\ge \ell+\frac{\Delta}{5}.
	\]
	On the other hand, the projection subproblem defining $x_j$ includes the cut constraint generated at $x_{j-1}$, and hence
	\[
		\ell_f(x_j;x_{j-1})\le \ell.
	\]
	Using the model-error bound \eqref{eq:model-error}, we obtain
	\[
		\frac{\Delta}{5}
		\le f(x_j)-\ell
		\le f(x_j)-\ell_f(x_j;x_{j-1})
		\le M\|x_{j-1}-x_j\|.
	\]
	Thus
	\begin{equation}\label{eq:step-lower-bound}
		\|x_{j-1}-x_j\|^2 \ge \frac{\Delta^2}{25M^2}.
	\end{equation}
	Combining \eqref{eq:step-lower-bound} with the three-point inequality \eqref{eq:three-point-step} gives
	\[
		\|\bar x-x_j\|^2
		\ge \|\bar x-x_{j-1}\|^2 + \frac{\Delta^2}{25M^2}.
	\]
	Summing this inequality from $j=1$ to $t$ yields
	\[
		\|\bar x-x_t\|^2
		\ge \|\bar x-x_0\|^2 + \frac{t\Delta^2}{25M^2}.
	\]
	Dividing by $\Delta^2$ and using the descent-slowness identity $s(a^{(t)},\bar x,\Delta)=\|\bar x-x_t\|/\Delta$ proves \eqref{eq:slowness-progress}. If the algorithm terminates at Line~\ref{line:blw-improve} of Algorithm~\ref{alg:blw}, then the final point itself triggers the improvement test and therefore need not satisfy the preceding non-improvement inequality. Applying \eqref{eq:slowness-progress} to the previously queried points gives \eqref{eq:slowness-progress-final}.

	We next prove the termination bound. If the algorithm reaches Line~\ref{line:blw-valid} at iteration $t$ and has not returned by the slowness test, then
	\[
		s(a^{(t)},\bar x,\Delta)^2 < \frac{2}{\mu\Delta}.
	\]
	At this point, the progress through $x_{t-1}$ has already been established,
	whereas the improvement test at $x_t$ has not yet been performed. Thus
	\eqref{eq:slowness-progress}, the monotonicity of the projection distances,
	and the nonnegativity of the initial descent-slowness square give
	\[
		\frac{t-1}{25M^2}
		\le s(a^{(t-1)},\bar x,\Delta)^2
		\le s(a^{(t)},\bar x,\Delta)^2
		<\frac{2}{\mu\Delta}.
	\]
	Therefore $t-1<50M^2/(\mu\Delta)$, which yields the stated bound of
	$50M^2/(\mu\Delta)+2$ oracle evaluations, including the oracle evaluation
	at $x_0$.

	We now prove the claimed gap bound whenever the algorithm returns with $\iscert{}=\texttt{True}$. If the algorithm returns at Line~\ref{line:blw-initial-cert}, then no point in $X$ satisfies $a^{(0)}(x)\le \ell$. Since $a^{(0)}$ is an affine minorant of $f$, this implies $f^*\ge \ell=f(\bar x)-\Delta$, and hence $f(\bar x)-f^*\le \Delta$. If the algorithm returns at Line~\ref{line:blw-loop-cert}, then, as discussed above, $c_t\ge \ell$ and $\psi_t\le f$, so again $f^*\ge c_t\ge \ell$ and $f(\bar x)-f^*\le \Delta$. Finally, if the algorithm returns at Line~\ref{line:blw-valid}, then $a^{(t)}$ is a $(\mu,\Delta)$ W-certificate, and Corollary~\ref{cor:valid-gap} yields
	\[
		f(\bar x)-f^*
		\le \Delta\vee \frac{\mu\Delta}{\mu^*}.
	\]

	Finally, suppose $f(\bar x)-f^*\le 4\Delta/5$.   Line~\ref{line:blw-improve} of Algorithm~\ref{alg:blw} cannot be triggered, because $f(x_t)\ge f^*\ge f(\bar x)-4\Delta/5$ for all $x_t\in X$. By the termination result just proved, the algorithm must therefore terminate with $\iscert{}=\texttt{True}$.
\end{proof}

\subsection{The BLW algorithm}\label{subsec:outer_blw_nonsmooth}
Next, we combine the BLCI subroutine with an outer guess-and-check scheme to obtain our bundle-level W-certificate
method (BLW), a method that adapts to the unknown QG modulus for nonsmooth functions. The method maintains an evaluation center $\bar x_k$, a candidate modulus $\mu_k$, an optimality gap estimator $\Delta_k$, and an affine minorant $a^{[k]}$. Here, $\Delta_k$ should be viewed as a guessed optimality gap to be certified, rather than as an a priori valid upper bound on the optimality gap.
\begin{algorithm}[H]
	\small
	\caption{The Bundle-Level W-Certificate Method $\mathrm{BLW}(\bar x_0,\mu_0,m)$}\label{alg:outer}
	\begin{algorithmic}[1]
		\Require Initial point $\bar x_0$, initial modulus estimate $\mu_0\ge \mu^*$, and memory length parameter $m\ge 1$.
		\Ensure A certified sequence $(\bar x_k,\mu_k,\Delta_k,a^{[k]})$.
		\State\label{line:outer-init}\begin{minipage}[t]{0.88\linewidth}
			Query the oracle at $\bar x_0$ and set
			\[
				\Delta_0 := \frac{2\|f'(\bar x_0)\|^2}{\mu_0}.
			\]
		\end{minipage}
		\For{$k=1,2,3,\ldots$}\label{line:outer-for}
		\State\label{line:outer-set} Set $\hat\mu_k:=\mu_{k-1}$ and $\hat\Delta_k:=\frac{3}{4}\Delta_{k-1}$.
		\State\label{line:outer-first-blw}\begin{minipage}[t]{0.84\linewidth}
			Run the first certify-or-improve call
			\[
				(a_{k,1},y_{k,1},\iscert{}_{k,1})\leftarrow
				\mathrm{BLCI}(\bar x_{k-1},\hat\mu_k,\hat\Delta_k,m,\ell_f(\cdot;\bar x_{k-1})).
			\]
		\end{minipage}
		\State\label{line:outer-first-cert} If $\iscert{}_{k,1}$ is \texttt{True}, set $\bar x_k:=\bar x_{k-1}$, $\mu_k:=\hat\mu_k$, $\Delta_k:=\hat\Delta_k$, $a^{[k]}:=a_{k,1}$, and continue to the next outer iteration $k+1$.
		\State\label{line:outer-second-blw}\begin{minipage}[t]{0.84\linewidth}
			Discard $y_{k,1}$ and run the safety certify-or-improve call
			\[
				(a_{k,2},z_k,\iscert{}_{k,2})\leftarrow
				\mathrm{BLCI}(\bar x_{k-1},\hat\mu_k,4\hat\Delta_k/5,m,a_{k,1}).
			\]
		\end{minipage}
		\State\label{line:outer-second-cert} If $\iscert{}_{k,2}$ is \texttt{True}, set $\bar x_k:=\bar x_{k-1}$, $\mu_k:=\hat\mu_k$, $\Delta_k:=\hat\Delta_k$, $a^{[k]}:=a_{k,2}$, and continue to the next outer iteration $k+1$.
		\State\label{line:outer-search-start}\begin{minipage}[t]{0.84\linewidth}
			Otherwise, accept the safe improved point $\bar x_k:=z_k$, set $\mu_{k,0}:=\hat\mu_k$, $\Delta_{k,0}:=\hat\Delta_k$, and $a_0^{[k]}(\cdot):=\ell_f(\cdot;\bar x_k)$.
		\end{minipage}
		\For{$r=1,2,3,\ldots$}\label{line:outer-search-for}
		\State\label{line:outer-embedded-blw}\begin{minipage}[t]{0.80\linewidth}
			Compute
			\begin{equation}\label{eq:embedded-blw}
				(a_r^{[k]},x_{k,r},\iscert{}_{k,r})\leftarrow
				\mathrm{BLCI}(\bar x_k,\mu_{k,r-1},\Delta_{k,r-1},m,a_{r-1}^{[k]}).
			\end{equation}
		\end{minipage}
		\State\label{line:outer-embedded-cert}\begin{minipage}[t]{0.80\linewidth}
			If $\iscert{}_{k,r}$ is \texttt{True}, set
			\[
				\mu_k:=\mu_{k,r-1},\qquad
				\Delta_k:=\Delta_{k,r-1},\qquad
				a^{[k]}:=a_r^{[k]},
			\]
			and continue to the next outer iteration $k+1$.
		\end{minipage}
		\State\label{line:outer-update-mu} Otherwise, discard $x_{k,r}$, set $\mu_{k,r}:=\mu_{k,r-1}/2$ and $\Delta_{k,r}:=2\Delta_{k,r-1}$, and continue to iteration $r+1$.
		\EndFor
		\EndFor
	\end{algorithmic}
\end{algorithm}
The algorithm takes any starting point $\bar x_0$ and any initial QG estimate $\mu_0\ge \mu^*$ as input. If such an upper estimate is unavailable, one can obtain a conservative initialization by querying two points $\hat x_1, \hat x_2 \in X$  with $f(\hat x_1)\ne f(\hat x_2)$ and taking
\begin{equation}\label{eq:mu0-def}
	\mu_0
	=\frac{2\max\{\|f'(\hat x_1)\|^2,\|f'(\hat x_2)\|^2\}}{|f(\hat x_1)-f(\hat x_2)|},
	\qquad
	\bar x_0 \in \argmax_{x\in\{\hat x_1,\hat x_2\}} f(x).
\end{equation}
For convex functions, QG implies Polyak-Łojasiewicz inequality~\cite[Corollary~6(ii)]{bolte2017error}, so $\mu_0$  is an upper bound on the growth modulus  $\mu^*$.

Two important observations will be used throughout the analysis. First, the center sequence generated by Algorithm~\ref{alg:outer} has nonincreasing objective values: a center is either left unchanged, or it is replaced by the point returned by the safety certify-or-improve call in Line~\ref{line:outer-second-blw}. Thus Assumption~\ref{ass:qg}, although stated only on the initial level set, applies to every center appearing in the certificate analysis below. Second,  as shown in the Lemma below, the design of Lines~\ref{line:outer-first-blw}--~\ref{line:outer-search-start} ensures that the distance to solution set is  nonincreasing.

\begin{lemma}\label{lem:safe-two-pass}
	Consider the following two-call safety check:
	\[
		(a_1,y_1,\iscert{}_1)
		\leftarrow
		\mathrm{BLCI}(\bar x,\mu,\Delta,m,a^{(0)}).
	\]
	If $\iscert{}_1=\texttt{False}$, run
	\[
		(a^+,x^+,\iscert{}_2)
		\leftarrow
		\mathrm{BLCI}(\bar x,\mu,4\Delta/5,m,a_1).
	\]
	If either call finishes with a true certificate flag, then the original pair $(\bar x,\mu,\Delta)$ is certified. If both calls return false, then
	\[
		f(x^+)<f(\bar x)-\frac{3}{5}\Delta, \quad  \dist(x^+,X^*)\le \dist(\bar x,X^*).
	\]
\end{lemma}

\begin{proof}
	If the first call returns with $\iscert{}_1=\texttt{True}$, the true-flag argument in the proof of Proposition~\ref{prop:blw-basic} certifies the pair $(\bar x,\mu,\Delta)$. If the first call returns false and the second call returns with $\iscert{}_2=\texttt{True}$, then the second call certifies the smaller gap $4\Delta/5$. This also certifies the original gap $\Delta$ because the right-hand side of~\eqref{eq:valid-gap} is monotone in $\Delta$.

	It remains to consider the case $\iscert{}_1=\iscert{}_2=\texttt{False}$. The first false return gives $f(y_1)<f(\bar x)-4\Delta/5$. Since $f^*\le f(y_1)$, we obtain $f^*<f(\bar x)-4\Delta/5$. The second call uses the level $\ell_2=f(\bar x)-4\Delta/5$. For every affine minorant $h$ used in that call and every $x^*\in X^*$, we have
	\[
		h(x^*)\le f(x^*)=f^*<\ell_2.
	\]
	Thus $X^*$ is contained in every projection set generated in the second call. Since the second call returns false, $x^+$ is one of the projection points generated by that call. Applying Lemma~\ref{lem:projection-three-point} with $\hat x=x^*$ gives
	\[
		\|x^+-x^*\|^2+\|x^+-\bar x\|^2\le \|\bar x-x^*\|^2,\qquad x^*\in X^*.
	\]
	Taking the infimum over $x^*\in X^*$ yields the distance bound. Finally, the second call uses gap $4\Delta/5$, so its false-return test guarantees a decrease by the factor $(4/5)^2$, which is larger than $3/5$. Thus $f(x^+)<f(\bar x)-3\Delta/5$.
\end{proof}

We now formally establish the convergence guarantees of Algorithm~\ref{alg:outer}. The next proposition is the main bookkeeping step. It keeps the proof entirely in terms of the descent-slowness $s(a,\bar x,\Delta)$ and the product $\mu\Delta$.

\begin{proposition}\label{prop:outer-step}
	Suppose that Assumptions~\ref{ass:model-error} and~\ref{ass:qg} hold. Suppose that at the beginning of the $k$th outer iteration the current center $\bar x_{k-1}$ is certified with parameters $(\mu_{k-1},\Delta_{k-1})$ and that $\mu_{k-1}\ge \mu^*/4$. Then the following properties hold for the $k$th outer iteration of Algorithm~\ref{alg:outer}.
	\begin{enumerate}[label=\arabic*.,ref=\arabic*,leftmargin=2em]
		\item\label{item:outer-step-certification} The output parameters satisfy
		      \begin{equation}\label{eq:product-contract}
			      \mu_k\Delta_k = \frac{3}{4}\mu_{k-1}\Delta_{k-1}.
		      \end{equation}
		      The returned center $\bar x_k$ is certified with parameters $(\mu_k,\Delta_k)$ with $\mu_k\ge \mu^*/4$.

		\item\label{item:outer-step-fejer} The accepted centers are Fej\'er monotone with respect to $X^*$:
		      \[
			      \dist(\bar x_k,X^*)\le \dist(\bar x_{k-1},X^*).
		      \]

		\item\label{item:outer-step-oracle-cost} The iteration takes at most the following number of oracle evaluations:
		      \begin{equation}\label{eq:outer-cost}
			      \frac{325M^2}{2\mu_k\Delta_k}+6+2\log_2\left(\frac{\mu_{k-1}}{\mu_k}\right).
		      \end{equation}
	\end{enumerate}
\end{proposition}

\begin{proof}
	By the objective-monotonicity observation above, all ordinary BLCI calls invoked in this outer iteration are centered at points in the initial level set. Hence Proposition~\ref{prop:blw-basic} applies to each of them.

	We first prove the certification and product claims. If either the first or the safety BLCI call returns with a true certificate flag, then Lemma~\ref{lem:safe-two-pass} certifies $\bar x_{k-1}$ with parameters $(\hat\mu_k,\hat\Delta_k)$, and Algorithm~\ref{alg:outer} sets $(\bar x_k,\mu_k,\Delta_k)=(\bar x_{k-1},\hat\mu_k,\hat\Delta_k)$. Hence
	\[
		\mu_k\Delta_k=\hat\mu_k\hat\Delta_k=\frac{3}{4}\mu_{k-1}\Delta_{k-1}.
	\]

	It remains to consider the case in which both calls return false. Then the algorithm accepts the point returned by the safety call as the new center and enters the embedded $\mu$-search. Whenever an embedded BLCI call returns with $\iscert{}=\texttt{True}$, Item~\ref{item:blw-cert-flag-gap} of Proposition~\ref{prop:blw-basic} certifies the fixed center $\bar x_k$ and the parameters used in that call. The embedded search preserves the product because each preceding false-flag step replaces $(\mu,\Delta)$ by $(\mu/2,2\Delta)$. Hence, once the certificate flag is true,
	\[
		\mu_k\Delta_k=\hat\mu_k\hat\Delta_k=\frac{3}{4}\mu_{k-1}\Delta_{k-1}.
	\]

	We now show that the output modulus is at least $\mu^*/4$. If the center is not moved, then $\mu_k=\mu_{k-1}\ge\mu^*/4$. Suppose instead that both calls return false. Lemma~\ref{lem:safe-two-pass} implies
	\[
		f(\bar x_k)<f(\bar x_{k-1})-\frac{3}{5}\hat\Delta_k.
	\]
	If $\mu_{k-1}\le\mu^*$, the previous certification gives $f(\bar x_{k-1})-f^*\le\Delta_{k-1}=\frac43\hat\Delta_k$. Therefore
	\[
		f(\bar x_k)-f^*<\frac{11}{15}\hat\Delta_k
		<\frac{4}{5}\hat\Delta_k.
	\]
	Item~\ref{item:blw-small-gap-cert} of Proposition~\ref{prop:blw-basic} then forces the first embedded BLCI call to return with a true certificate flag, and $\mu_k=\mu_{k-1}\ge\mu^*/4$.

	Now suppose $\mu_{k-1}>\mu^*$. Let $j$ be the smallest nonnegative integer such that
	\[
		2^j\hat\Delta_k\ge \frac{5}{4}\bigl(f(\bar x_k)-f^*\bigr).
	\]
	The embedded BLCI call with parameters $(2^{-j}\hat\mu_k,2^j\hat\Delta_k)$ returns with $\iscert{}=\texttt{True}$ by Item~\ref{item:blw-small-gap-cert}. If $j\ge1$, minimality of $j$ and the previous certification imply
	\[
		2^{j-1}\hat\Delta_k
		<\frac{5}{4}\bigl(f(\bar x_k)-f^*\bigr)
		\le \frac{5}{4}\bigl(f(\bar x_{k-1})-f^*\bigr)
		\le \frac{5}{3}\frac{\mu_{k-1}}{\mu^*}\hat\Delta_k.
	\]
	Thus $2^j<\frac{10}{3}\mu_{k-1}/\mu^*<4\mu_{k-1}/\mu^*$. Since the algorithm obtains a true certificate flag no later than this embedded call, $\mu_k\ge 2^{-j}\mu_{k-1}>\mu^*/4$. The case $j=0$ is immediate.

	The Fej\'er monotonicity claim is immediate if the center is not moved. If the center is moved, then both calls return false, and Lemma~\ref{lem:safe-two-pass} gives
	\[
		\dist(\bar x_k,X^*)\le \dist(\bar x_{k-1},X^*).
	\]
	The embedded search never changes the center after a false flag; it only updates $\mu$, $\Delta$, and the affine warm start.

	It remains to count oracle calls. The two-pass safety check consists of one BLCI call at product $\hat\mu_k\hat\Delta_k$, and, only if the first call returns false, one additional BLCI call at product $4\hat\mu_k\hat\Delta_k/5$. Therefore, its cost is at most
	\[
		\frac{50M^2}{\hat\mu_k\hat\Delta_k}+2+\frac{125M^2}{2\hat\mu_k\hat\Delta_k}+2
	\]
	oracle evaluations. If the safety check returns with a true certificate flag, the desired bound follows immediately. Otherwise, the embedded search is entered. Let $\bar r$ be the index of the embedded BLCI call that first returns with $\iscert{}_{k,\bar r}=\texttt{True}$, so $\bar r=1+\log_2(\mu_{k-1}/\mu_k)$. Let $N_{\mathrm{ter},r}$ denote the last inner iteration of the $r$th embedded call, with the convention $N_{\mathrm{ter},r}=0$ if the call returns in its initial test, and define
	\[
		n_r:=(N_{\mathrm{ter},r}-1)_+.
	\]
	Thus, $n_r$ counts only iterations completed before the terminal step. For every false-flag call, Item~\ref{item:blw-slowness-progress} of Proposition~\ref{prop:blw-basic} shows that the  square of the descent-slowness increases by at least $n_r/(25M^2)$. The same estimate holds through the last pre-exit aggregate of the final true-flag call. Since the gap doubles after each false flag, Lemma~\ref{lem:monotonicity} carries these increases across the warm starts. Consequently,
	\[
		\frac{1}{25M^2}\sum_{r=1}^{\bar r}n_r
		<\frac{2}{\hat\mu_k\hat\Delta_k}.
	\]
	The $r$th embedded call uses at most $n_r+2$ oracle evaluations. Therefore
	\[
		\sum_{r=1}^{\bar r}(n_r+2)
		\le \frac{50M^2}{\hat\mu_k\hat\Delta_k}+2\bar r
		= \frac{50M^2}{\hat\mu_k\hat\Delta_k}+2+2\log_2\left(\frac{\mu_{k-1}}{\mu_k}\right).
	\]
	Adding the possible safety-call overhead gives
	\[
		N_k
		\le \frac{325M^2}{2\hat\mu_k\hat\Delta_k}
		+6+2\log_2\left(\frac{\mu_{k-1}}{\mu_k}\right).
	\]
	Since $\hat\mu_k\hat\Delta_k=\mu_k\Delta_k$, this proves \eqref{eq:outer-cost}.
\end{proof}

Taken together, we recover the following oracle complexity bound for finding an $\epsilon$-optimal solution for any $\epsilon>0$.

\begin{theorem}\label{thm:anytime-complexity}
	Suppose that Assumptions~\ref{ass:model-error} and~\ref{ass:qg} hold. Let $\mu_0\ge \mu^*$ and let $(\bar x_k,\mu_k,\Delta_k)_{k\ge 0}$ be generated by Algorithm~\ref{alg:outer}. For any $\epsilon>0$, define
	\begin{equation}\label{eq:sepsilon}
		S_\epsilon
		:= \max\left\{0,\left\lceil\log_{4/3}\left(\frac{4\mu_0\Delta_0}{\mu^*\epsilon}\right)\right\rceil\right\}
		\le \max\left\{0,\left\lceil\log_{4/3}\left(\frac{8\|f'(\bar x_0)\|^2}{\mu^*\epsilon}\right)\right\rceil\right\}.
	\end{equation}
	Then
	\begin{equation}\label{eq:epsilon-optimality}
		f(\bar x_{S_\epsilon})-f^*\le \epsilon.
	\end{equation}
	The center sequence is Fej\'er monotone with respect to $X^*$:
	\begin{equation*}
		\dist(\bar x_k,X^*)\le \dist(\bar x_{k-1},X^*),\qquad k\ge 1.
	\end{equation*}
	Moreover, the number of first-order oracle evaluations used up to the construction of $\bar x_{S_\epsilon}$ is at most
	\begin{equation}\label{eq:total-complexity}
		N_\epsilon
		\le \frac{3500M^2}{\mu^*\epsilon}
		+6S_\epsilon
		+2\log_2\left(\frac{4\mu_0}{\mu^*}\right)
		+1.
	\end{equation}
\end{theorem}

\begin{proof}
	The upper bound on $S_\epsilon$ follows from the initial choice of $\Delta_0$ in Algorithm~\ref{alg:outer}, which gives $\mu_0\Delta_0=2\|f'(\bar x_0)\|^2$.
	We first verify the base case required for the outer induction. Let
	$\delta_0:=f(\bar x_0)-f^*$. Convexity and Assumption~\ref{ass:qg} give
	\[
		\delta_0
		\le \|f'(\bar x_0)\|\dist(\bar x_0,X^*)
		\le \|f'(\bar x_0)\|\sqrt{\frac{2\delta_0}{\mu^*}}.
	\]
	Consequently, $\delta_0\le2\|f'(\bar x_0)\|^2/\mu^*$, and hence
	\[
		f(\bar x_0)-f^*
		\le \frac{\mu_0}{\mu^*}\Delta_0
		\le \max\left\{1,\frac{\mu_0}{\mu^*}\right\}\Delta_0.
	\]
	Thus, the initial center is certified with parameters
	$(\mu_0,\Delta_0)$.

	Applying Proposition~\ref{prop:outer-step} inductively gives, for all $k\ge 0$,
	\begin{equation}\label{eq:inductive-product}
		\mu_k\Delta_k = \left(\frac{3}{4}\right)^k\mu_0\Delta_0,
		\qquad
		\mu_k\ge \frac{\mu^*}{4},
	\end{equation}
	and the current center $\bar x_k$ is certified with parameters $(\mu_k,\Delta_k)$. Item~\ref{item:outer-step-fejer} of the same proposition gives the claimed Fej\'er monotonicity. Therefore the gap bound from Section~\ref{sec:affine_certificate} yields
	\begin{equation}\label{eq:gap-from-product}
		\begin{aligned}
			f(\bar x_k)-f^*
			 & \le \max\left\{1,\frac{\mu_k}{\mu^*}\right\}\Delta_k \\
			 & \le \frac{4\mu_k\Delta_k}{\mu^*}.
		\end{aligned}
	\end{equation}
	By the definition of $S_\epsilon$, we have $\mu_{S_\epsilon}\Delta_{S_\epsilon}\le \mu^*\epsilon/4$. Substituting this into \eqref{eq:gap-from-product} proves $f(\bar x_{S_\epsilon})-f^*\le \epsilon$.

	It remains to sum the oracle costs. The initialization in Line~\ref{line:outer-init} uses one oracle evaluation at $\bar x_0$. If $S_\epsilon=0$, this is the only evaluation and the claimed bound follows. Hence, assume $S_\epsilon\ge1$. Since $\mu_{S_\epsilon}\ge \mu^*/4$ by Item~\ref{item:outer-step-certification} of Proposition~\ref{prop:outer-step}, summing Item~\ref{item:outer-step-oracle-cost} of Proposition~\ref{prop:outer-step} and including the initial oracle  evaluation at $\bar x_0$ gives
	\begin{equation}\label{eq:oracle-sum}
		N_\epsilon
		\le \frac{325M^2}{2}\sum_{k=1}^{S_\epsilon}\frac{1}{\mu_k\Delta_k}
		+6S_\epsilon
		+2\log_2\left(\frac{4\mu_0}{\mu^*}\right)+1.
	\end{equation}
	It follows from the definition of $S_\epsilon$ that
	\[
		\mu_{S_\epsilon}\Delta_{S_\epsilon}
		> \mu_0\Delta_0\left(\frac{3}{4}\right)^{\log_{4/3}\left(\frac{4\mu_0\Delta_0}{\mu^*\epsilon}\right)+1}
		= \frac{3\mu^*\epsilon}{16}.
	\]
	Thus we obtain
	\[
		\sum_{k=1}^{S_\epsilon}\frac{1}{\mu_k\Delta_k}
		= \frac{1}{\mu_{S_\epsilon}\Delta_{S_\epsilon}}
		\sum_{j=0}^{S_\epsilon-1}\left(\frac{3}{4}\right)^j
		\le \frac{4}{\mu_{S_\epsilon}\Delta_{S_\epsilon}}
		< \frac{64}{3\mu^*\epsilon}.
	\]
	Substituting this into \eqref{eq:oracle-sum} yields
	\[
		N_\epsilon
		\le \frac{3500M^2}{\mu^*\epsilon}
		+6S_\epsilon
		+2\log_2\left(\frac{4\mu_0}{\mu^*}\right)+1,
	\]
	which is \eqref{eq:total-complexity}.
\end{proof}

Theorem~\ref{thm:anytime-complexity} provides the optimal anytime guarantee: the algorithm does not require $\epsilon$ as an input, yet the leading term of the complexity bound achieves the optimal nonsmooth scaling $M^2/(\mu^*\epsilon)$. Additionally, the observable quantity $\mu_k \Delta_k$ provides a reliable proxy for the optimality gap and can therefore be used as a practical stopping criterion. We discuss this point further in Section~\ref{sec:numerics}.

\section{The uniformly optimal A-BLW method}\label{sec:ablw-holder}

In this section, we present the accelerated bundle-level certify-or-improve (A-BLCI) subroutine for H\"older smooth convex objectives and then use it inside the same outer search to obtain the accelerated bundle-level W-certificate (A-BLW) method. The subroutine has the same input-output convention as BLCI in Algorithm~\ref{alg:blw}: it takes an affine minorant as input and returns an affine minorant, a point, and the certificate flag $\iscert{}$. The only algorithmic difference is that A-BLCI evaluates cuts along an accelerated search sequence.
\begin{algorithm}[H]
	\small
	\caption{\small The Accelerated Bundle-Level Certify-or-Improve Method $\mathrm{A\text{-}BLCI}(\bar x,\mu,\Delta,m,a^{(0)})$}\label{alg:ablw}
	\begin{algorithmic}[1]
		\Require Evaluation center $\bar x$, QG modulus estimate $\mu$, gap estimate $\Delta$, memory length $m$, and affine minorant $a^{(0)}$.
		\Ensure An affine minorant $a^+$,  a point $x^+$ that is either $\bar x$ or a potentially improved solution, and a certificate flag $\iscert{}$.
		\State\label{line:ablw-initial-cert} Set $\ell=f(\bar x)-\Delta$. If $\{x\in X:a^{(0)}(x)\le \ell\}=\emptyset$, return $(a^{(0)},\bar x,\texttt{True})$.
		\State\label{line:ablw-set-xzero} Set
		\Statex \hspace{\algorithmicindent}\begin{minipage}{.9\linewidth}
			\begin{equation}
				x_0\in\argmin_{x\in X}\left\{\norm{x-\bar x}^2:a^{(0)}(x)\le \ell\right\}.
			\end{equation}
		\end{minipage}
		\State Initialize $\bar y_0=y_0=x_0$.\label{line:ablw-initialize-y}
		\For{$t=1,2,3,\ldots$}\label{line:ablw-loop}
		\State Query the oracle at $y_{t-1}$ to evaluate $\ell_f(\cdot;y_{t-1})$.\label{line:ablw-query-y}
		\State For the max-model\label{line:ablw-model}
		\Statex \hspace{\algorithmicindent}\begin{minipage}{.9\linewidth}
			\[
				\psi_t(x):=a^{(t-1)}(x)\vee \max_{[t-m]_+\le i\le t-1}\ell_f(x;y_i).
			\]
		\end{minipage}
		\State\label{line:ablw-loop-cert} Set $c_t:=\inf_{x\in X}\psi_t(x)$. If $c_t\ge \ell$, set $a_{c_t}(\cdot)\equiv c_t$ and return $(a_{c_t},\bar x,\texttt{True})$.
		\State Solve \label{line:ablw-projection}
		\Statex \hspace{\algorithmicindent}\begin{minipage}{.9\linewidth}
			\begin{equation}
				\begin{aligned}
					x_t\in\argmin_{x\in X}\quad & \norm{x-\bar x}^2                                 \\
					\st\quad                    & a^{(t-1)}(x)\le \ell,                             \\
					                            & \ell_f(x;y_i)\le \ell,\qquad [t-m]_+\le i\le t-1.
				\end{aligned}
			\end{equation}
		\end{minipage}
		\State Extract dual multipliers from the previous problem and normalize them to obtain $\lambda_a\in\R_+$ and $\lambda_i\in\R_+$ with $\lambda_a+\sum_i\lambda_i=1$. Set\label{line:ablw-aggregate}
		\Statex \hspace{\algorithmicindent}\begin{minipage}{.9\linewidth}
			\[
				a^{(t)}(\cdot):=\lambda_a a^{(t-1)}(\cdot)+\sum_{i=[t-m]_+}^{t-1}\lambda_i\ell_f(\cdot;y_i).
			\]
		\end{minipage}
		\State If $\norm{x_t-\bar x}^2/\Delta^2\ge 2/(\mu\Delta)$, return $(a^{(t)},\bar x,\texttt{True})$.\label{line:ablw-valid-exit}
		\State With $\alpha_t=2/(t+1)$, set $\bar y_t=\alpha_t x_t+(1-\alpha_t)\bar y_{t-1}$ and $y_t=\alpha_{t+1}x_t+(1-\alpha_{t+1})\bar y_t$.\label{line:ablw-update-y}
		\State\label{line:ablw-improvement-exit} Query the oracle at $\bar y_t$.  If $f(\bar y_t)<f(\bar x)-4\Delta/5$, return $(a^{(t)},\bar y_t,\texttt{False})$.
		\EndFor
	\end{algorithmic}
\end{algorithm}
We now formally introduce the more general Hölder smoothness condition for the problems studied in this section.
\begin{assumption}[H\"older smoothness]\label{ass:holder-smooth-ablw}
	The function $f$ is $(\rho,M_\rho)$-H\"older smooth on $X$ for some $\rho\in[0,1]$ and $M_\rho>0$, namely, for every $x,y\in X$ and every $f'(x)\in\partial f(x)$,
	\begin{equation}\label{eq:holder-smooth-ablw}
		f(y)\le f(x)+\inner{f'(x)}{y-x}+\frac{M_\rho}{1+\rho}\norm{y-x}^{1+\rho}.
	\end{equation}
\end{assumption}
Since achieving the optimal oracle complexity for smooth optimization requires acceleration, Algorithm~\ref{alg:ablw} augments BLCI with the Nesterov-type estimate sequence from \cite{lan2015bundle,nesterov2003introductory}. The cutting planes are evaluated at the search sequence $\{y_t\}$, while the function-value decrease test uses the estimate sequence $\{\bar y_t\}$. Apart from these additional sequences, the certification logic is the same as in Algorithm~\ref{alg:blw}: a true certificate flag either follows from the affine W-certificate test or from a direct lower bound $c_t\ge \ell$, which implies that
\begin{equation*}
	f(\bar x)-f^* \le \Delta \vee \frac{\mu\Delta}{\mu^*}.
\end{equation*}

The following lemma characterizes the projection geometry of the iterates. Its proof is identical to that of Lemma~\ref{lem:projection-three-point} and is therefore omitted.
\begin{lemma}[Projection three-point inequality]\label{lem:ablw-threepoint}
	Consider iterates $\{x_t\}$ and aggregate models $a^{(t)}$ generated by Algorithm~\ref{alg:ablw}. Each $a^{(t)}$ is an affine minorant of $f$. For every aggregate model generated before termination,
	\begin{equation}\label{eq:ablw-three-point}
		\norm{\bar x-x_t}^2+\norm{x_t-\hat x}^2\le \norm{\hat x-\bar x}^2,\qquad \forall \hat x\in\{x\in X:a^{(t)}(x)\le \ell\}.
	\end{equation}
	Consequently, whenever $x_{t+1}$ is generated,
	\begin{equation}\label{eq:ablw-three-point-step}
		\norm{x_{t+1}-x_t}^2+\norm{x_t-\bar x}^2\le \norm{x_{t+1}-\bar x}^2.
	\end{equation}
\end{lemma}

With that, the next proposition analyzes how the descent-slowness $s(a^{(t)},\bar x,\Delta)$ grows in each iteration of an A-BLCI call. It also shows that the certificate flag plays the same role as it does in Algorithm~\ref{alg:blw}.

\begin{proposition}\label{prop:ablw}
	Suppose that Assumptions~\ref{ass:qg} and~\ref{ass:holder-smooth-ablw} hold, and suppose that the evaluation center $\bar x$ lies in the initial level set from Assumption~\ref{ass:qg}. Consider one call to Algorithm~\ref{alg:ablw} with parameters $(\bar x,\mu,\Delta)$, memory length $m \ge 1$, and initial affine minorant $a^{(0)}$. Let $N_{\mathrm{ter}}$ denote the last iteration, and define
	\[
		Q_\rho(\mu,\Delta):=\left(\frac{M_\rho^2}{\mu^{1+\rho}\Delta^{1-\rho}}\right)^{\frac{1}{1+3\rho}}.
	\]
	Then the following properties hold for the iterates generated by the A-BLCI algorithm.
	\begin{enumerate}[label=\arabic*.,ref=\arabic*,leftmargin=2em]
		\item\label{item:ablw-slowness-progress} For every generated main iterate $x_t$ with $t<N_{\mathrm{ter}}$,
		      \begin{equation}\label{eq:ablw-slowness-progress}
			      s(a^{(t)},\bar x,\Delta)^2\ge s(a^{(0)},\bar x,\Delta)^2+\frac{1}{100}\left(\frac{t^{1+3\rho}}{M_\rho^2\Delta^{2\rho}}\right)^{\frac{1}{1+\rho}}.
		      \end{equation}
		      Consequently, if Algorithm~\ref{alg:ablw} terminates at Line~\ref{line:ablw-improvement-exit} with a significantly improved solution, the associated affine minorant satisfies
		      \begin{equation}\label{eq:ablw-slowness-progress-final}
			      s(a^{(N_{\mathrm{ter}})},\bar x,\Delta)^2\ge s(a^{(0)},\bar x,\Delta)^2+\frac{1}{100}\left(\frac{(N_{\mathrm{ter}}-1)^{1+3\rho}}{M_\rho^2\Delta^{2\rho}}\right)^{\frac{1}{1+\rho}}.
		      \end{equation}
		\item\label{item:ablw-termination} The A-BLCI algorithm terminates within $400Q_\rho(\mu,\Delta)+4$ oracle evaluations for every $\rho\in[0,1]$.
		\item\label{item:ablw-cert-flag-gap} When the A-BLCI algorithm returns with $\iscert{}=\texttt{True}$ for the pair $(\bar x,\mu,\Delta)$, we have
		      \[
			      f(\bar x)-f^*\le \Delta\vee \frac{\mu\Delta}{\mu^*}.
		      \]
		\item\label{item:ablw-small-gap-cert} If $f(\bar x)-f^*\le 4\Delta/5$, the A-BLCI algorithm returns with $\iscert{}=\texttt{True}$, and hence certifies the pair $(\bar x,\mu,\Delta)$.
	\end{enumerate}
\end{proposition}

\begin{proof}
	As in the proof of Proposition~\ref{prop:blw-basic}, each aggregate model remains an affine minorant of $f$, and whenever the sublevel set is nonempty,
	\[
		d(a^{(t)},\bar x,\Delta)=\norm{\bar x-x_t},
		\qquad
		s(a^{(t)},\bar x,\Delta)=\frac{\norm{\bar x-x_t}}{\Delta}.
	\]
	Fix $t<N_{\mathrm{ter}}$. We first relate the away step $d_j:=\norm{x_j-x_{j-1}}$ to the function value gap $f(\bar y_j)-\ell$ for each $1\le j\le t$. By the update rule in Line~\ref{line:ablw-update-y}, we have
	\[
		\bar y_j-y_{j-1}=\alpha_j(x_j-x_{j-1}).
	\]
	By the $(\rho,M_\rho)$-H\"older smoothness of $f$,
	\[
		\begin{aligned}
			f(\bar y_j)-\ell
			 & \le \ell_f(\bar y_j;y_{j-1})-\ell+\frac{M_\rho}{1+\rho}\norm{\bar y_j-y_{j-1}}^{1+\rho}                                       \\
			 & =(1-\alpha_j)\ell_f(\bar y_{j-1};y_{j-1})+\alpha_j\ell_f(x_j;y_{j-1})-\ell+\frac{M_\rho}{1+\rho}\alpha_j^{1+\rho}d_j^{1+\rho} \\
			 & \le (1-\alpha_j)(f(\bar y_{j-1})-\ell)+\frac{M_\rho}{1+\rho}\alpha_j^{1+\rho}d_j^{1+\rho},
		\end{aligned}
	\]
	where the last inequality uses the convexity of $f$ and the projection-step constraint $\ell_f(x_j;y_{j-1})\le \ell$ from Line~\ref{line:ablw-projection}. Multiplying the preceding recursion by $\omega_j=j(j+1)/2$, using $\omega_{j-1}=\omega_j(1-\alpha_j)$, and summing from $j=1$ to $t$ gives a telescoping cancellation of $\{f(\bar y_j)-\ell\}$ and yields
	\[
		\omega_t[f(\bar y_t)-\ell]\le \frac{M_\rho}{1+\rho}\sum_{j=1}^t \omega_j\alpha_j^{1+\rho}d_j^{1+\rho}.
	\]
	Since $t<N_{\mathrm{ter}}$, the improvement test in Line~\ref{line:ablw-improvement-exit} has not been passed at $\bar y_t$, so $f(\bar y_t)-\ell\ge \Delta/5$. Hence
	\begin{equation}\label{eq:ablw-holder-step-sum}
		\frac{t^2\Delta}{10}\le \frac{M_\rho}{1+\rho}\sum_{j=1}^t \omega_j\alpha_j^{1+\rho}d_j^{1+\rho}.
	\end{equation}
	Next, since
	\[
		\omega_j\alpha_j^{1+\rho}=\frac{j(j+1)}{2}\left(\frac{2}{j+1}\right)^{1+\rho}=2^\rho\frac{j}{(j+1)^\rho}\le 2^\rho j^{1-\rho},
	\]
	H\"older's inequality gives
	\[
		\begin{aligned}
			\sum_{j=1}^t j^{1-\rho}d_j^{1+\rho}
			 & \le \left(\sum_{j=1}^t d_j^2\right)^{\frac{1+\rho}{2}}
			\left(\sum_{j=1}^t j^2\right)^{\frac{1-\rho}{2}}                                  \\
			 & \le \left(\sum_{j=1}^t d_j^2\right)^{\frac{1+\rho}{2}}t^{\frac{3(1-\rho)}{2}}.
		\end{aligned}
	\]
	Substituting these bounds into \eqref{eq:ablw-holder-step-sum} and rearranging yields
	\[
		\sum_{j=1}^t d_j^2\ge \left(\frac{1+\rho}{5\cdot 2^{\rho+1}}\right)^{\frac{2}{1+\rho}}
		\left(\frac{\Delta^2}{M_\rho^2}\right)^{\frac{1}{1+\rho}}
		t^{\frac{1+3\rho}{1+\rho}}.
	\]
	Simple calculus shows that $\left(\frac{1+\rho}{5\cdot 2^{\rho+1}}\right)^{\frac{2}{1+\rho}}$ is at least $(1/10)^2=1/100$ when $\rho \in [0,1]$. Combining this estimate with the three-point inequality \eqref{eq:ablw-three-point-step}, we obtain the following lower bound after $t$ iterations:
	\[
		\norm{x_t-\bar x}^2\ge \norm{x_0-\bar x}^2+\sum_{j=1}^t d_j^2
		\ge \norm{x_0-\bar x}^2+\frac{1}{100}\left(\frac{\Delta^2}{M_\rho^2}\right)^{\frac{1}{1+\rho}}t^{\frac{1+3\rho}{1+\rho}}.
	\]
	Dividing both sides by $\Delta^2$ proves \eqref{eq:ablw-slowness-progress}. If the algorithm terminates at Line~\ref{line:ablw-improvement-exit}, then the final estimate point may trigger the improvement test, so applying the preceding estimate to the previously generated main iterates gives \eqref{eq:ablw-slowness-progress-final}.

	We next prove termination. If the algorithm reaches Line~\ref{line:ablw-valid-exit} at iteration $t$ and has not returned by the slowness test, then
	\[
		s(a^{(t)},\bar x,\Delta)^2<\frac{2}{\mu\Delta}.
	\]
	The progress estimate available before the improvement test at $\bar y_t$
	extends only through $x_{t-1}$. Therefore
	\eqref{eq:ablw-slowness-progress}, the monotonicity of the projection
	distances, and the nonnegativity of $s(a^{(0)},\bar x,\Delta)^2$ give
	\[
		\left(\frac{(t-1)^{1+3\rho}}{M_\rho^2\Delta^{2\rho}}\right)^{\frac{1}{1+\rho}}
		<\frac{200}{\mu\Delta}.
	\]
	Equivalently, every nonterminal iteration index $t$ satisfies
	\[
		t-1<200^{\frac{1+\rho}{1+3\rho}}Q_\rho(\mu,\Delta)
		\le 200Q_\rho(\mu,\Delta),
	\]
	so the last iteration is at most $\lceil 200Q_\rho(\mu,\Delta)\rceil+1$. Since each iteration makes at most two oracle evaluations, the total number of oracle evaluations is at most
	\[
		2\lceil 200Q_\rho(\mu,\Delta)\rceil+2
		\le 400Q_\rho(\mu,\Delta)+4,
	\]
	which proves Item~\ref{item:ablw-termination}.

	The proofs of Items~\ref{item:ablw-cert-flag-gap} and~\ref{item:ablw-small-gap-cert} are identical to that of Proposition~\ref{prop:blw-basic}, so we omit them for brevity.
\end{proof}
The next lemma is the A-BLCI counterpart of Lemma~\ref{lem:safe-two-pass}; it will be used to establish Fej\'er monotonicity of the iterates.
\begin{lemma}\label{lem:accelerated-safe-two-pass}
	Consider the following two-call accelerated safety check:
	\[
		(a_1,u_1,\iscert{}_1)
		\leftarrow
		\mathrm{A\text{-}BLCI}(\bar x,\mu,\Delta,m,a^{(0)}).
	\]
	If $\iscert{}_1=\texttt{False}$, run
	\[
		(a^+,y^+,\iscert{}_2)
		\leftarrow
		\mathrm{A\text{-}BLCI}(\bar x,\mu,4\Delta/5,m,a_1).
	\]
	If the safety check stops with a true certificate flag, then the original pair $(\bar x,\mu,\Delta)$ is certified. If both calls return false, then
	\[
		f(y^+)<f(\bar x)-\frac{3}{5}\Delta,
		\qquad
		\dist(y^+,X^*)\le \dist(\bar x,X^*).
	\]
\end{lemma}

\begin{proof}
	The true-flag case is identical to Lemma~\ref{lem:safe-two-pass}: a true flag in the first call certifies the original gap, while a true flag in the second call certifies the smaller gap $4\Delta/5$, which also certifies the original gap $\Delta$ by monotonicity of~\eqref{eq:valid-gap}.

	Now suppose both calls return false. The first false return gives $f(u_1)<f(\bar x)-4\Delta/5$, and hence $f^*<f(\bar x)-4\Delta/5$. Therefore the second call uses the level $\ell_2=f(\bar x)-4\Delta/5>f^*$. Thus, every minimizer is feasible for every projection subproblem in the second call. Fix $x^*\in X^*$. Lemma~\ref{lem:ablw-threepoint} gives $\|x_t-x^*\|\le\|\bar x-x^*\|$ for every projection point $x_t$ generated in that call. Also, $\bar y_0=x_0$ has the same property. Since $\bar y_t=\alpha_t x_t+(1-\alpha_t)\bar y_{t-1}$ and Euclidean balls are convex, induction gives $\|\bar y_t-x^*\|\le\|\bar x-x^*\|$ for every generated estimate point $\bar y_t$. Since the second call returns false, $y^+=\bar y_t$ for some $t$, and taking the infimum over $x^*\in X^*$ yields the distance bound. The objective decrease follows from the false-return test in the second call and the inequality $(4/5)^2>3/5$.
\end{proof}

Having established the basic properties of the A-BLCI subroutine, we are ready to introduce our  uniformly optimal algorithm.
\subsection{Optimal complexity under H\"older smoothness and quadratic growth}\label{sec:QG_Holder}

Since Algorithms~\ref{alg:ablw} and~\ref{alg:blw} have the same input--output interface, A-BLCI can replace BLCI in Algorithm~\ref{alg:outer}. We refer to the resulting method as \emph{A-BLW}; its pseudocode is omitted because it differs from Algorithm~\ref{alg:outer} only in this subroutine substitution. We establish its uniform optimality for problems with  H\"older smoothness and quadratic growth. Recall that, up to a universal constant factor,  the oracle complexity of Algorithm~\ref{alg:ablw} with input $\Delta$, for a convex objective satisfying quadratic growth with modulus $\mu$ and the H\"older smoothness condition with exponent $\rho\in[0,1]$ and constant $M_\rho$,  is
\begin{equation}
	Q_\rho(\mu,\Delta)=\left(\frac{M_\rho^2}{\mu^{1+\rho}\Delta^{1-\rho}}\right)^{\frac{1}{1+3\rho}}
\end{equation}
Lemma~\ref{lem:accelerated-safe-two-pass} ensures both objective monotonicity and Fej\'er monotonicity of the accepted centers. Hence, all centers remain in the initial level set where Assumption~\ref{ass:qg} holds.
The next result is the accelerated counterpart of Proposition~\ref{prop:outer-step}. It shows that one outer iteration produces a valid certificate for the next center, reduces the product $\mu_k\Delta_k$ by a constant factor, and uses $O(Q_\rho(\mu_k,\Delta_k)+\log(\mu_{k-1}/\mu_k))$ oracle evaluations.

\begin{proposition}[One outer iteration with A-BLW]\label{prop:outer-ablw}
	Suppose that Assumptions~\ref{ass:qg} and~\ref{ass:holder-smooth-ablw} hold. Suppose that, at the beginning of the $k$th outer iteration of A-BLW, the current center $\bar x_{k-1}$ admits a valid certificate with parameters $(\mu_{k-1},\Delta_{k-1})$ and $\mu_{k-1}\ge \mu^*/4$. Then the output parameters satisfy
	\begin{equation}\label{eq:ablw-product-contract}
		\mu_k\Delta_k=\frac34\mu_{k-1}\Delta_{k-1},
	\end{equation}
	and the returned center $\bar x_k$ admits a valid certificate with parameters $(\mu_k,\Delta_k)$ satisfying $\mu_k\ge \mu^*/4$. Moreover,
	\[
		\dist(\bar x_k,X^*)\le \dist(\bar x_{k-1},X^*).
	\]
	If $N_k$ denotes the number of first-order oracle evaluations used in this outer iteration, then
	\begin{equation}\label{eq:ablw-outer-cost}
		N_k\le 2900Q_\rho(\mu_k,\Delta_k)+10+2\log_2\left(\frac{\mu_{k-1}}{\mu_k}\right).
	\end{equation}
\end{proposition}

\begin{proof}
	The product identity, certificate validity, lower bound $\mu_k\ge \mu^*/4$, and Fej\'er monotonicity follow exactly as in Proposition~\ref{prop:outer-step}, using Lemma~\ref{lem:accelerated-safe-two-pass} in place of Lemma~\ref{lem:safe-two-pass}.

	It remains only to count oracle calls. During the embedded search, the product governing the descent-slowness threshold remains fixed at $\mu_k\Delta_k$. Moreover, within the embedded search the modulus estimate can only decrease, so the quantity $Q_\rho(\mu,\Delta)$ for any earlier call is no larger than $Q_\rho(\mu_k,\Delta_k)$ for the final certified pair. By Item~\ref{item:ablw-termination} of Proposition~\ref{prop:ablw}, one A-BLCI call with parameters $(\mu,\Delta)$ uses at most $400Q_\rho(\mu,\Delta)+4$ oracle evaluations for every $\rho\in[0,1]$. The initial two-pass safety check consists of one call at $(\hat\mu_k,\hat\Delta_k)$ and, only if the first call returns false, one additional call at $(\hat\mu_k,4\hat\Delta_k/5)$. Since $Q_\rho(\hat\mu_k,4\hat\Delta_k/5)\le (5/4)Q_\rho(\mu_k,\Delta_k)$, the safety check costs at most
	\[
		400Q_\rho(\mu_k,\Delta_k)+4+500Q_\rho(\mu_k,\Delta_k)+4.
	\]

	If the initial two-pass safety check returns with a true certificate flag, then \eqref{eq:ablw-outer-cost} follows immediately. Hence suppose that the embedded $\mu$-search is entered. Let $\bar r$ denote the first embedded index whose A-BLCI call returns with a true certificate flag, so that
	\[
		\Delta_{k,r-1}=2^{r-1}\widehat\Delta_k,
		\qquad
		\mu_{k,r-1}=2^{-(r-1)}\widehat\mu_k,
		\qquad r=1,\ldots,\bar r,
	\]
	and $\Delta_k=\Delta_{k,\bar r-1}$, $\mu_k=\mu_{k,\bar r-1}$. Thus
	\[
		\bar r=1+\log_2\left(\frac{\mu_{k-1}}{\mu_k}\right).
	\]
	Let $N_r$ be the last inner iteration of the $r$th embedded A-BLCI call, with $N_r=0$ if it exits in its initial test, and set $n_r:=(N_r-1)_+$. Thus $n_r$ counts only the main iterates generated before the terminal step. Since each inner iteration makes at most two oracle queries, this call uses at most $2n_r+2$ oracle evaluations. We claim the following bound for the embedded search:
	\begin{equation}\label{eq:embedded-ablw-count}
		\sum_{r=1}^{\bar r}n_r\le 1000Q_\rho(\mu_k,\Delta_k),\qquad 0\le \rho\le1.
	\end{equation}
	Indeed, Item~\ref{item:ablw-slowness-progress} of Proposition~\ref{prop:ablw} gives, for every false-flag call and through the last pre-exit aggregate of the final true-flag call, a descent-slowness-square increase of at least
	\[
		\frac{1}{100}
		\left(\frac{n_r^{1+3\rho}}{M_\rho^2\Delta_{k,r-1}^{2\rho}}\right)^{\frac{1}{1+\rho}}.
	\]
	The affine minorant is warm-started across false-flag calls, the gap doubles, and the product remains $\mu_k\Delta_k$. Hence Lemma~\ref{lem:monotonicity} telescopes these increases. As in the nonaccelerated proof, the reference aggregate has already failed the slowness test and therefore has descent-slowness square below $2/(\mu_k\Delta_k)$. If the final call exits before producing a pre-exit aggregate, then $n_{\bar r}=0$ and we use the preceding false-flag output at its own gap; when $\bar r=1$, the sum is zero. We obtain
	\[
		\sum_{r=1}^{\bar r}\left(\frac{n_r^{1+3\rho}}{M_\rho^2\Delta_{k,r-1}^{2\rho}}\right)^{\frac{1}{1+\rho}}
		< \frac{200}{\mu_k\Delta_k}.
	\]
	To convert this moment bound into a bound on the total number of inner iterations, index the stages backward from the final certified call. Set $a_i:=n_{\bar r-i}$, $i=0,\ldots,\bar r-1$. Since $\Delta_{k,\bar r-i-1}=2^{-i}\Delta_k$, the moment bound first gives
	\[
		\sum_{i=0}^{\bar r-1}2^{i\frac{2\rho}{1+\rho}}
		a_i^{\frac{1+3\rho}{1+\rho}}<B,
		\qquad
		B:=\frac{200M_\rho^{2/(1+\rho)}\Delta_k^{2\rho/(1+\rho)}}{\mu_k\Delta_k}.
	\]
	In particular,
	\[
		\sum_{i=0}^{\bar r-1}a_i^{\frac{1+3\rho}{1+\rho}}\le B,
		\qquad
		a_i^{\frac{1+3\rho}{1+\rho}}\le 2^{-i\frac{2\rho}{1+\rho}}B.
	\]
	Applying Lemma~\ref{lem:moment} then gives
	\[
		\sum_{r=1}^{\bar r}n_r=\sum_{i=0}^{\bar r-1}a_i\le 5B^{\frac{1+\rho}{1+3\rho}}\le 1000Q_\rho(\mu_k,\Delta_k).
	\]
	Thus the embedded search uses at most
	$2000Q_\rho(\mu_k,\Delta_k)+2\bar r$ oracle evaluations. Combining this with the safety-check overhead of at most $900Q_\rho(\mu_k,\Delta_k)+8$ and $\bar r=1+\log_2(\mu_{k-1}/\mu_k)$ yields \eqref{eq:ablw-outer-cost}.
\end{proof}
Now we present the complexity guarantee of A-BLW.
\begin{theorem}[Optimal complexity of A-BLW]\label{thm:ablw-complexity}
	Suppose that Assumptions~\ref{ass:qg} and~\ref{ass:holder-smooth-ablw} hold. Let $\mu_0\ge\mu^*$, and let $(\bar x_k,\mu_k,\Delta_k)_{k\ge0}$ be generated by A-BLW. For any $\epsilon>0$, define
	\begin{equation}\label{eq:ablw-sepsilon}
		S_\epsilon:=\max\left\{0,\left\lceil \log_{4/3}\left(\frac{4\mu_0\Delta_0}{\mu^*\epsilon}\right)\right\rceil\right\}
		\le \max\left\{0,\left\lceil \log_{4/3}\left(\frac{8\norm{f'(\bar x_0)}^2}{\mu^*\epsilon}\right)\right\rceil\right\}.
	\end{equation}
	Then $f(\bar x_{S_\epsilon})-f^*\le \epsilon$. The center sequence is Fej\'er monotone with respect to $X^*$:
	\[
		\dist(\bar x_k,X^*)\le \dist(\bar x_{k-1},X^*),\qquad k\ge1.
	\]
	Moreover, if $N_\epsilon$ denotes the total number of first-order oracle evaluations required to complete the $S_\epsilon$th outer iteration, then there is a universal constant $C$ such that
	\begin{equation}\label{eq:ablw-total-complexity}
		N_\epsilon\le C\min\left\{\frac{1}{1-\rho},S_\epsilon\right\}
		\left(\frac{M_\rho^2}{(\mu^*)^{1+\rho}\epsilon^{1-\rho}}\right)^{\frac{1}{1+3\rho}}
		+10S_\epsilon+2\log_2\left(\frac{4\mu_0}{\mu^*}\right)+1.
	\end{equation}
	At the endpoint $\rho=1$, the minimum is interpreted as $S_\epsilon$.
\end{theorem}

\begin{proof}
	Since the outer logic is unchanged and Proposition~\ref{prop:outer-ablw} provides the same product contraction and modulus lower bound as in the nonaccelerated case, the proof of Theorem~\ref{thm:anytime-complexity} applies verbatim to show $f(\bar x_{S_\epsilon})-f^*\le \epsilon$ and Fej\'er monotonicity. If $S_\epsilon=0$, only the initial oracle evaluation at $\bar x_0$ is needed, so the complexity bound follows immediately. Hence assume $S_\epsilon\ge1$. The same product bookkeeping used in the proof of Theorem~\ref{thm:anytime-complexity} then gives
	\[
		\mu_{S_\epsilon}\Delta_{S_\epsilon}>\frac{3}{16}\mu^*\epsilon.
	\]
	It remains to bound the number of oracle evaluations.

	Including the initial oracle evaluation at $\bar x_0$ and applying Proposition~\ref{prop:outer-ablw}, we obtain
	\begin{equation}\label{eq:ablw-oracle-sum}
		N_\epsilon\le 2900\sum_{k=1}^{S_\epsilon}Q_\rho(\mu_k,\Delta_k)+10S_\epsilon+2\sum_{k=1}^{S_\epsilon}\log_2\left(\frac{\mu_{k-1}}{\mu_k}\right)+1.
	\end{equation}
	The logarithmic sum telescopes:
	\[
		\sum_{k=1}^{S_\epsilon}\log_2\left(\frac{\mu_{k-1}}{\mu_k}\right)
		=\log_2\left(\frac{\mu_0}{\mu_{S_\epsilon}}\right)
		\le \log_2\left(\frac{4\mu_0}{\mu^*}\right),
	\]
	where the last inequality uses $\mu_{S_\epsilon}\ge \mu^*/4$.

	It remains to estimate the main sum. For each $1\le k\le S_\epsilon$, using $\mu_k\ge\mu^*/4$, $\mu_k\Delta_k=(3/4)^k\mu_0\Delta_0$, and the lower bound on $\mu_{S_\epsilon}\Delta_{S_\epsilon}$ above, we have
	\[
		\begin{aligned}
			Q_\rho(\mu_k,\Delta_k)
			 & =\left(\frac{M_\rho^2}{\mu_k^{2\rho}(\mu_k\Delta_k)^{1-\rho}}\right)^{\frac{1}{1+3\rho}}              \\
			 & \le 4^{\frac{2\rho}{1+3\rho}}\left(\frac{16}{3}\right)^{\frac{1-\rho}{1+3\rho}}Q_\rho(\mu^*,\epsilon)
			\left(\left(\frac34\right)^{\frac{1-\rho}{1+3\rho}}\right)^{S_\epsilon-k}                                \\
			 & \le \frac{16}{3}Q_\rho(\mu^*,\epsilon)
			\left(\left(\frac34\right)^{\frac{1-\rho}{1+3\rho}}\right)^{S_\epsilon-k}.
		\end{aligned}
	\]
	Consequently, the elementary geometric bound
	\[
		\sum_{k=1}^{S_\epsilon}\left(\left(\frac34\right)^{\frac{1-\rho}{1+3\rho}}\right)^{S_\epsilon-k}
		\le \min\left\{\frac{16}{1-\rho},S_\epsilon\right\}
	\]
	implies
	\[
		\sum_{k=1}^{S_\epsilon}Q_\rho(\mu_k,\Delta_k)
		\le \frac{256}{3}\min\left\{\frac{1}{1-\rho},S_\epsilon\right\}Q_\rho(\mu^*,\epsilon).
	\]
	Substituting this estimate and the telescoping logarithmic bound into \eqref{eq:ablw-oracle-sum}, and absorbing numerical factors into the universal constant $C$, gives \eqref{eq:ablw-total-complexity}.
\end{proof}

\begin{remark}
	The complexity bound on $N_\epsilon$ shows that the proposed A-BLW method attains a universally optimal rate without requiring prior knowledge of the true QG modulus $\mu^*$. Specifically, when $\rho=0$, the leading term reduces to $O(M_0^2/(\mu^*\epsilon))$. For any fixed $\rho\in(0,1)$, since $\min\{1/(1-\rho),S_\epsilon\} \le \frac{1}{1-\rho}$, the leading complexity bound becomes
	\[
		O\left(\left(\frac{M_\rho^2}{(\mu^*)^{1+\rho}\epsilon^{1-\rho}}\right)^{\frac{1}{1+3\rho}}\right),
	\]
	matching the corresponding lower bounds \cite{doikov2025lower}. As $\rho \rightarrow 1$, $\min\{1/(1-\rho),S_\epsilon\}\le S_\epsilon=O(\log(1/\epsilon))$. Consequently, at the smooth endpoint $\rho=1$, the bound becomes the accelerated smooth linear rate $O(\sqrt{M_1/\mu^*}\log(1/\epsilon))$.
\end{remark}

\subsection{Universally optimal rate of A-BLW}
\label{subsec:other-growth-conditions}

Recall that $f$ satisfies the $\alpha$-H\"older growth condition on the
initial level set if
\[
	f(x)-f^*
	\ge
	\mu_\alpha \dist^\alpha(x,X^*),
	\qquad
	\text{for all }x\in \{z\in X:f(z)\le f(\bar x_0)\},
\]
where $\bar x_0$ is the initialization. In this subsection, we show that
A-BLW is adaptive to both the smoothness and growth regimes. Namely, for every
H\"older smoothness setting considered here, A-BLW attains the corresponding
optimal rate under either $\alpha$-H\"older growth, for any $\alpha\ge2$, or
under no growth assumption beyond convexity. No smoothness parameter, growth
exponent, or growth constant is assumed to be known. The key step is to reduce
each case to a suitable quadratic-growth estimate along the generated
trajectory and then invoke the results of Section~\ref{sec:QG_Holder}.

We begin with a trajectory-level version of quadratic growth. The proof of
Theorem~\ref{thm:ablw-complexity} uses quadratic growth only at the generated
centers $\{\bar x_k\}$. Hence, for a fixed $K$, the relevant modulus is the
sequential quadratic-growth modulus
\begin{equation}\label{eq:sequential-qg-modulus}
	\begin{aligned}
		\mu_K^{\rm seq}
		 & :=
		\sup\left\{
		\mu\ge0:
		f(\bar x_j)-f^*
		\ge
		\frac{\mu}{2}\dist^2(\bar x_j,X^*)
		\quad \forall\,0\le j\le K
		\right\} \\
		 & =
		\inf_{0\le j\le K}
		\frac{2(f(\bar x_j)-f^*)}{\dist^2(\bar x_j,X^*)},
	\end{aligned}
\end{equation}
where the ratio is interpreted as $+\infty$ at optimal points.

This observation allows Theorem~\ref{thm:ablw-complexity} to be applied with a
pre-stopping modulus. Fix $\epsilon>0$, and let
\[
	\tau_\epsilon:=\inf\{k\ge0:f(\bar x_k)-f^*\le\epsilon\}
\]
be the first outer iteration at which an $\epsilon$-optimal center is obtained,
with the convention that $\tau_\epsilon=+\infty$ if no such iteration exists.
If $\tau_\epsilon=0$, there is nothing to prove. Thus assume
$\tau_\epsilon>0$. Define
\[
	\widehat\mu_\epsilon
	:=
	\inf_{0\le j<\tau_\epsilon}
	\frac{2(f(\bar x_j)-f^*)}{\dist^2(\bar x_j,X^*)},
\]
again interpreting the ratio as $+\infty$ at optimal points. When
$\tau_\epsilon<+\infty$, this is exactly
$\mu_{\tau_\epsilon-1}^{\rm seq}$.

We next verify that the initialization in~\eqref{eq:mu0-def} gives
$\widehat\mu_\epsilon\le\mu_0$. Since
$\tau_\epsilon>0$, we have $\bar x_0\notin X^*$. By convexity,
\[
	f(\bar x_0)-f^*
	\le
	\inner{f'(\bar x_0)}{\bar x_0-x^*}
	\le
	\norm{f'(\bar x_0)}\dist(\bar x_0,X^*),
\]
where $x^*\in X^*$ is a projection of $\bar x_0$ onto $X^*$. Therefore
\[
	\begin{aligned}
		\widehat\mu_\epsilon
		 & \le
		\mu_0^{\rm seq}
		\le
		\frac{2\norm{f'(\bar x_0)}^2}{f(\bar x_0)-f^*} \\
		 & \le
		\frac{
			2\max\{\norm{f'(\hat x_1)}^2,\norm{f'(\hat x_2)}^2\}
		}{
			|f(\hat x_1)-f(\hat x_2)|
		}.
	\end{aligned}
\]
By~\eqref{eq:mu0-def}, the last expression is exactly $\mu_0$, and hence
$\widehat\mu_\epsilon\le\mu_0$. Let $\widehat S_\epsilon$ be the quantity
$S_\epsilon$ from
Theorem~\ref{thm:ablw-complexity} with $\mu^*$ replaced by
$\widehat\mu_\epsilon$. If $\widehat\mu_\epsilon>0$, the stopping-time
argument below gives $\tau_\epsilon\le\widehat S_\epsilon$ and bounds the
oracle calls needed to find an $\epsilon$-optimal center by the estimate in
Theorem~\ref{thm:ablw-complexity} with $\mu^*$ replaced by
$\widehat\mu_\epsilon$.

Indeed, suppose the method had not reached accuracy $\epsilon$ by iteration
$\widehat S_\epsilon$. Then $\widehat S_\epsilon<\tau_\epsilon$, so all
centers $\bar x_0,\ldots,\bar x_{\widehat S_\epsilon}$ are included in the
definition of $\widehat\mu_\epsilon$. These centers therefore satisfy the
quadratic-growth inequality with modulus $\widehat\mu_\epsilon$. Since the proof of
Theorem~\ref{thm:ablw-complexity} uses quadratic growth only at these centers,
the same proof applies with $\mu^*$ replaced by $\widehat\mu_\epsilon$, and
yields
\[
	f(\bar x_{\widehat S_\epsilon})-f^*\le\epsilon,
\]
a contradiction.

\paragraph{The convex setting.}
Now suppose that no growth condition is imposed and $X^*\ne\emptyset$. Define
$R_0:=\dist(\bar x_0,X^*)$. If $R_0=0$, then $\bar x_0$ is already optimal. Thus
assume $R_0>0$ and $\tau_\epsilon>0$. By the definition of $\tau_\epsilon$
and the Fej\'er monotonicity in Theorem~\ref{thm:ablw-complexity}, for every
$0\le j<\tau_\epsilon$,
\[
	f(\bar x_j)-f^*>\epsilon,
	\qquad
	\dist(\bar x_j,X^*)\le R_0.
\]
Therefore, the stopping-time modulus satisfies
\[
	\widehat\mu_\epsilon
	\ge
	\frac{2\epsilon}{R_0^2}.
\]
Substituting this bound into the stopping-time estimate above gives the leading
complexity bound
\[
	N_\epsilon
	=
	O\left(
	\min\left\{
	\frac{1}{1-\rho},
	\log_{4/3}\left(\frac{\mu_0\Delta_0R_0^2}{\epsilon^2}\right)
	\right\}
	\left(
	\frac{M_\rho^2R_0^{2(1+\rho)}}{\epsilon^2}
	\right)^{\frac{1}{1+3\rho}}
	\right).
\]
For every fixed $\rho<1$, this is the optimal H\"older-smooth convex rate. At
the smooth endpoint $\rho=1$, the direct use of
Theorem~\ref{thm:ablw-complexity} gives an extra multiplicative logarithmic
factor if the per-outer-iteration cost $\sqrt{M_1/\mu_k}$ is bounded uniformly
by its worst value $\sqrt{M_1/\hat\mu_\epsilon}$. This logarithm is an artifact of that crude summation. To
remove it,  we instead sum up a tighter  sequence of geometric upper bounds. Corollary~\ref{cor:convex-mu-lower-bound}
shows that
\[
	\mu_k\ge\frac{\sqrt{\mu_k\Delta_k}}{\sqrt5\,R_0}.
\]
Hence
\[
	Q_1(\mu_k,\Delta_k)
	=
	\sqrt{\frac{M_1}{\mu_k}}
	\le
	5^{1/4}\sqrt{M_1}\,R_0^{1/2}(\mu_k\Delta_k)^{-1/4}.
\]
Using the product identity $\mu_k\Delta_k=(3/4)^k\mu_0\Delta_0$ and summing
through the stopping time gives
\[
	\sum_{k=1}^{\tau_\epsilon} Q_1(\mu_k,\Delta_k)
	\le
	C\sqrt{M_1}\,R_0^{1/2}
	(\mu_{\tau_\epsilon}\Delta_{\tau_\epsilon})^{-1/4}.
\]
Moreover, $\tau_\epsilon\le\widehat S_\epsilon$ and the definition of
$\widehat S_\epsilon$ give
\[
	\mu_{\tau_\epsilon}\Delta_{\tau_\epsilon}
	>
	\frac{3}{16}\widehat\mu_\epsilon\epsilon
	\ge
	\frac{3\epsilon^2}{8R_0^2}.
\]
Substituting this lower bound into the geometric-envelope sum gives the leading
smooth convex rate
\[
	N_\epsilon
	=
	O\left(\frac{\sqrt{M_1}\,R_0}{\sqrt{\epsilon}}\right).
\]

\paragraph{Convex setting with H\"older growth.}
Finally, we consider convex objectives satisfying an $\alpha$-H\"older growth condition for $\alpha \ge 2$.
Assume again that
$\tau_\epsilon>0$. For every $0\le j<\tau_\epsilon$, with
\[
	\delta_j:=f(\bar x_j)-f^*,
	\qquad
	d_j:=\dist(\bar x_j,X^*),
\]
we have $\delta_j>\epsilon$. By the objective monotonicity of the centers and the growth condition, we have
\[
	d_j^2
	\le
	\left(\frac{\delta_j}{\mu_\alpha}\right)^{2/\alpha}.
\]
Therefore
\[
	\frac{2\delta_j}{d_j^2}
	\ge
	2\mu_\alpha^{2/\alpha}\delta_j^{1-2/\alpha}
	\ge
	2\mu_\alpha^{2/\alpha}\epsilon^{(\alpha-2)/\alpha},
\]
and hence
\[
	\widehat\mu_\epsilon\ge
	2\mu_\alpha^{2/\alpha}\epsilon^{(\alpha-2)/\alpha}.
\]
Using the stopping-time estimate above gives
\[
	N_\epsilon
	=
	O\left(
	\min\left\{
	\frac{1}{1-\rho},
	\log_{4/3}\left(
		\frac{\mu_0\Delta_0}{\mu_\alpha^{2/\alpha}
			\epsilon^{2(\alpha-1)/\alpha}}
		\right)
	\right\}
	\left(
		\frac{M_\rho^2}{
			\mu_\alpha^{2(1+\rho)/\alpha}
			\epsilon^{2(\alpha-1-\rho)/\alpha}}
		\right)^{\frac{1}{1+3\rho}}
	\right).
\]
This recovers the growth-lifting rate of~\cite{grimmer2023holder}, while the
algorithm itself does not require $\mu_\alpha$, $\alpha$, $\rho$, $M_\rho$, or
$\epsilon$ as inputs. When $\rho=1$ and $\alpha>2$, the logarithmic factor in
the direct bound is again removable by summing the corresponding geometric
upper bounds. Corollary~\ref{cor:higher-order-mu-lower-bound} shows that
\[
	\mu_k\ge
	c_\alpha\,\mu_\alpha^{1/(\alpha-1)}
	(\mu_k\Delta_k)^{\frac{\alpha-2}{2(\alpha-1)}}
\]
for an explicit constant $c_\alpha>0$. Set
\[
	\vartheta_\alpha:=\frac{\alpha-2}{4(\alpha-1)}.
\]
Then, for a constant $C_\alpha$ depending only on $\alpha$,
\[
	Q_1(\mu_k,\Delta_k)
	\le
	C_\alpha\sqrt{M_1}\,
	\mu_\alpha^{-\frac{1}{2(\alpha-1)}}
	(\mu_k\Delta_k)^{-\vartheta_\alpha}.
\]
Using the geometric contraction of $\mu_k\Delta_k$ and summing through the
stopping time gives
\[
	\sum_{k=1}^{\tau_\epsilon} Q_1(\mu_k,\Delta_k)
	\le
	C_\alpha\sqrt{M_1}\,
	\mu_\alpha^{-\frac{1}{2(\alpha-1)}}
	(\mu_{\tau_\epsilon}\Delta_{\tau_\epsilon})^{-\vartheta_\alpha}.
\]
Moreover, $\tau_\epsilon\le\widehat S_\epsilon$ implies, for a universal
constant $c>0$,
\[
	\mu_{\tau_\epsilon}\Delta_{\tau_\epsilon}
	>
	\frac{3}{16}\widehat\mu_\epsilon\epsilon
	\ge
	c\mu_\alpha^{2/\alpha}
	\epsilon^{2(\alpha-1)/\alpha}.
\]
Substitution this bound yields, for each fixed $\alpha>2$, the oracle complexity is

\[
	N_\epsilon
	=
	O\left(
	\frac{\sqrt{M_1}}{
		\mu_\alpha^{1/\alpha}\epsilon^{(\alpha-2)/(2\alpha)}}
	\right).
\]

\section{Numerical experiments}
\label{sec:numerics}

The goal of this section is to test the performance of A-BLW across several convex problem settings.  All curves use the best point found so far:
\[
	f_{\mathrm{best}}(N)
	:=
	\min\{f(x_t): \;x_t
	\text{ has been queried within the first $N$ oracle calls}\}.
\]
The plotted quantity is the optimality gap $
	f_{\mathrm{best}}(N)-f^\star$,
where \(f^\star\) is computed independently for reporting.   Unless stated otherwise,
\(\mu_0\) is initialized automatically by the two-point rule~\eqref{eq:mu0-def}.

\subsection{Matrix games}
\label{subsec:matrix-games}

Our first test problem is the matrix-game example used by
Nesterov~\cite{nesterov2015universal}.  Given a payoff matrix
\(A\in\mathbb{R}^{n\times m}\), the classical saddle-point formulation is
\begin{equation}
	\min_{x\in\Delta_n}\max_{y\in\Delta_m} \langle x,Ay\rangle,
	\label{eq:matrix-game-saddle}
\end{equation}
where \(\Delta_d=\{u\in\mathbb{R}^d_+:\mathbf{1}^\top u=1\}\) is the standard simplex.
Since the inner maximization is over a simplex, an optimal \(y\) can always
be chosen as a coordinate vector.  Thus \eqref{eq:matrix-game-saddle} is
equivalent to the nonsmooth convex minimization problem
\begin{equation}
	\min_{x\in\Delta_n} f_A(x),
	\qquad
	f_A(x):=\max_{1\le j\le m} a_j^\top x,
	\label{eq:matrix-game-primal}
\end{equation}
where \(a_j=Ae_j\) is the \(j\)-th column of \(A\). We treat the
inner maximum in \eqref{eq:matrix-game-primal} as the objective function and
apply \(\mathrm{BLW}\), \(\mathrm{A\text{-}BLW}\), or Nesterov's universal fast gradient method
directly to \(f_A\) over the simplex.
The experiments use square matrix games with \(n=m=512\).  The entries of
\(A\) are sampled independently from the uniform distribution on
\([-1,1]\), matching the random matrix-game construction in
\cite{nesterov2015universal}.  Each curve is averaged over ten runs.

\paragraph{Effect of bundle memory.}
Figure~\ref{fig:matrix-cuts-average} compares \(\mathrm{A\text{-}BLW}\) with
\(m_{\mathrm{cut}}\in\{1,20,50\}\).  Increasing
the number of retained cuts improves the final gap for this example. As a side note, we also run BLW with $m=512$ and observed linear rate of convergence. The underlying reason is that the objective is piecewise smooth; when $m$ exceeds the number of pieces, BLW converges linearly. We refer interested readers to~\cite{zhang2025linearlyconvergentalgorithmsnonsmooth} for further details.

\begin{figure}[H]
	\centering
	\includegraphics[width=0.82\textwidth]{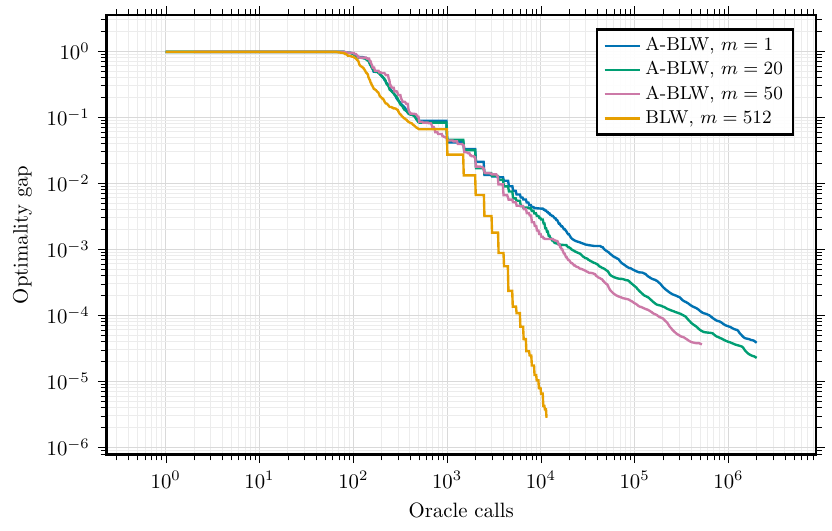}
	\caption{Matrix game with \(n=m=512\).  The plot shows oracle calls
		versus the optimality gap of the best point found so far.  Curves are
		averaged over ten random seeds.}
	\label{fig:matrix-cuts-average}
\end{figure}

\paragraph{Comparison with Nesterov's universal FGM.}
We next compare \(\mathrm{A\text{-}BLW}\) with \(m_{\mathrm{cut}}=1\) to Nesterov's universal fast
gradient method.  The FGM implementation is run on the same ten matrix-game
instances and uses the same primal oracle for \(f_A\).  We test the input
accuracy parameters
\[
	\varepsilon_{\texttt{target}} \in \{10^{-1},10^{-2},10^{-3},10^{-4}\}.
\]
The FGM parameter \(\varepsilon_{\texttt{target}} \) is the method's requested target accuracy~\cite{nesterov2015universal}, not
the true optimality gap used on the vertical axis.

Figure~\ref{fig:matrix-fgm-average} shows a visible sensitivity to
\(\varepsilon_{\texttt{target}} \).  On these instances, the best FGM curve among the tested
inputs is not obtained by the smallest \(\varepsilon_{\texttt{target}} \); choosing
\(\varepsilon\) too small requires excessive oracle calls since the complexity of FGM for nonsmooth problems is $O(\epsilon^{-2})$ even with quadratic growth.  The \(\mathrm{A\text{-}BLW}\) curve does not use
such an input accuracy and continues to reduce the best gap over the plotted
oracle-call budget.

\begin{figure}[H]
	\centering
	\includegraphics[width=0.82\textwidth]{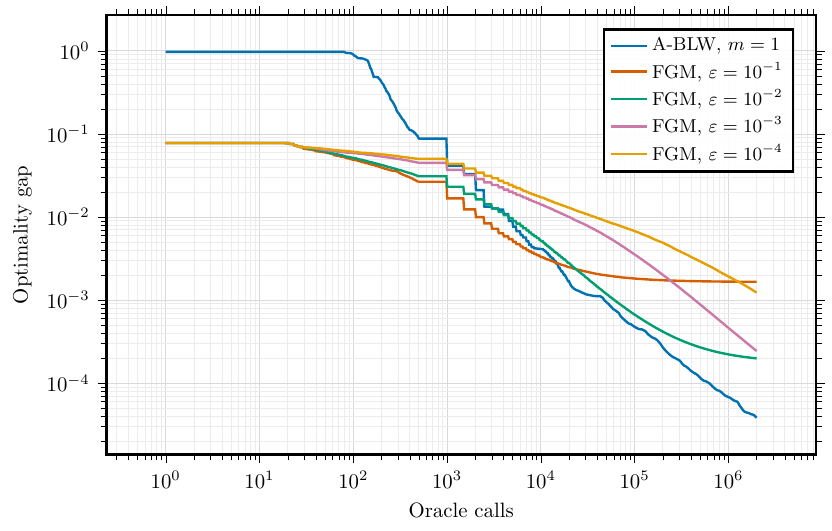}
	\caption{Matrix game with \(n=m=512\): \(\mathrm{A\text{-}BLW}\) with one cut compared
		with Nesterov's universal FGM for several input accuracies
		\(\varepsilon\).  Curves are averaged over the same ten random seeds.}
	\label{fig:matrix-fgm-average}
\end{figure}

\subsection{Geometric median}
\label{subsec:geometric-median}

The second family of experiments is the geometric median problem from
\cite{nesterov2015universal}.  Given centers
\(b_1,\ldots,b_M\in\mathbb{R}^d\), the problem is
\begin{equation}
	\min_{x\in Q} F(x),
	\qquad
	F(x):=\sum_{i=1}^M \|x-b_i\|_2,
	\label{eq:geometric-median}
\end{equation}
where \(Q\subseteq\mathbb{R}^d\) is a closed convex set.  Nesterov's numerical
example uses \(d=256\), \(M=512\), \(Q=\mathbb{R}^d_+\), and centers sampled
uniformly from the box
\[
	0\le b_i^{(k)}\le \frac{1}{\sqrt d},
	\qquad
	k=1,\ldots,d.
\]
The box has Euclidean diameter one, and taking \(M>d\) creates many
nonsmooth locations.

This problem is qualitatively different from the matrix game.  The objective
is still nonsmooth, but it is not a maximum of finitely many affine
functions.  Away from the centers it behaves like a smooth function, and the
nonsmoothness is localized at the data points.  This makes it a useful
second test: the matrix game stresses polyhedral nonsmoothness, while the
geometric median tests whether the bundle-level methods can exploit the
more benign local geometry of a distance-sum objective.

For this example, the true optimality gap $F(x)-F^\star$ is useful for assessing a run,
but it is not a practical termination rule because $F^\star$ is unknown. Figure~\ref{fig:geomed-cert-custom} also reports the scale-normalized product
\[
	\frac{\mu\Delta}{\omega_0},
	\qquad
	\omega_0:=\|g(\bar x_0)\|_2,
\]
where $\bar x_0$ is the auto-initialized starting center. This quantity is
included only as a heuristic termination criterion since the product $\mu\Delta$ is observable and, by
Corollary~\ref{cor:valid-gap}, constitutes the key term in the upper bound on the optimality gap. The normalization by $\omega_0$ is included solely to  ensure that the resulting metric has the same unit as the function gap.

We use \(d=256\), \(M=512\), and \(p=1\), with centers generated from the
box above using seed \(1\).  The initial point is the origin.  The initial
quadratic-growth estimate is chosen by the automatic two-point
initialization, using the origin and the opposite corner
\((1/\sqrt d)\mathbf{1}\).

\begin{figure}[H]
	\centering
	\includegraphics[width=0.82\textwidth]{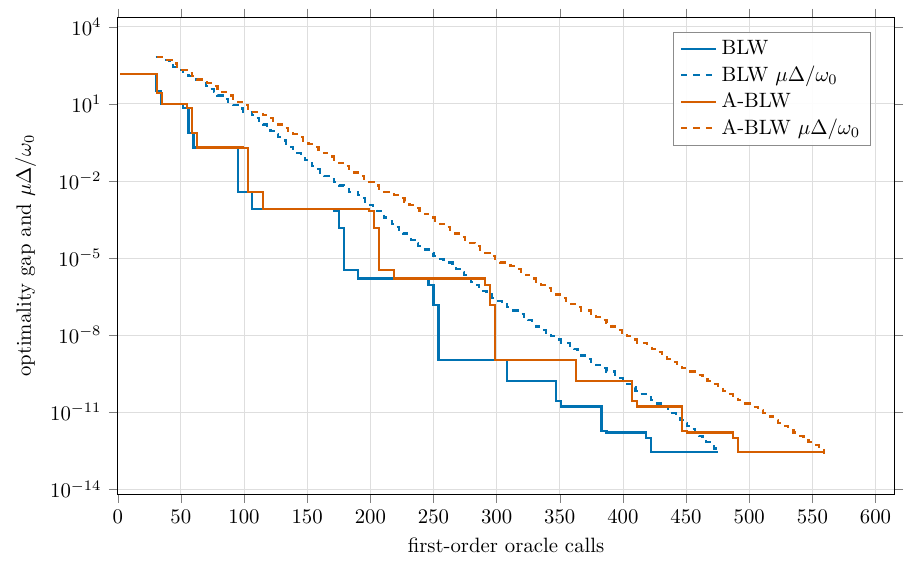}
	\caption{Geometric median instance with $d=256$, $M=512$, and one
		cut.  Solid lines show the true optimality gap of the best point found
		so far. Dashed lines show the heuristic scale-normalized product
		$\mu\Delta/\omega_0$; they are diagnostics, not certified gap bounds.}
	\label{fig:geomed-cert-custom}
\end{figure}

In this experiment, the normalized product decreases in tandem with the true
gap and is close to it near the end of the run. This empirical agreement is
useful for visualizing the method's progress, but it should not be interpreted
as a guarantee.
\subsection{Simplex-constrained quadratics}
\label{subsec:simplex-quadratics}

The final experiment isolates the effect of conditioning.  The objective is
smooth, so the purpose of the test is not to introduce another source of
nonsmoothness.  Rather, it asks whether the accelerated bundle-level
subroutine gives a genuine advantage as the condition number increases.  To this end, we
solve
\begin{equation}
	\min_{x\in\Delta_d}
	q(x)
	:=
	\frac12 \sum_{i=1}^d \lambda_i x_i^2,
	\qquad
	\Delta_d=\{x\in\mathbb{R}^d_+:\mathbf{1}^\top x=1\}.
	\label{eq:simplex-quadratic}
\end{equation}
All eigenvalues are positive, so the minimizer is in the relative interior
of the simplex.  The KKT conditions give the following closed-form solution:
\begin{equation}
	x_i^\star
	=
	\frac{\lambda_i^{-1}}{\sum_{j=1}^d \lambda_j^{-1}},
	\qquad
	q^\star
	=
	\frac{1}{2\sum_{j=1}^d \lambda_j^{-1}}.
	\label{eq:simplex-quadratic-solution}
\end{equation}
Thus the optimal value is known exactly and
can be used to stop the run once the best observed gap is below \(10^{-6}\).
We fix \(d=50\) and choose a geometric spectrum
\[
	\lambda_i
	=
	\kappa^{(i-1)/(d-1)},
	\qquad
	i=1,\ldots,d,
\]
where
$\kappa\in
	\{10,30,100,300,10^3,3\cdot10^3,10^4,3\cdot10^4,
	10^5,3\cdot10^5,10^6\}$.
The horizontal axis in Figure~\ref{fig:simplex-condition} is the restricted condition number
on the simplex.  Both
methods use \(m_{\mathrm{cut}}=1\) and the same initialization.

\begin{figure}[H]
	\centering
	\includegraphics[width=0.82\textwidth]{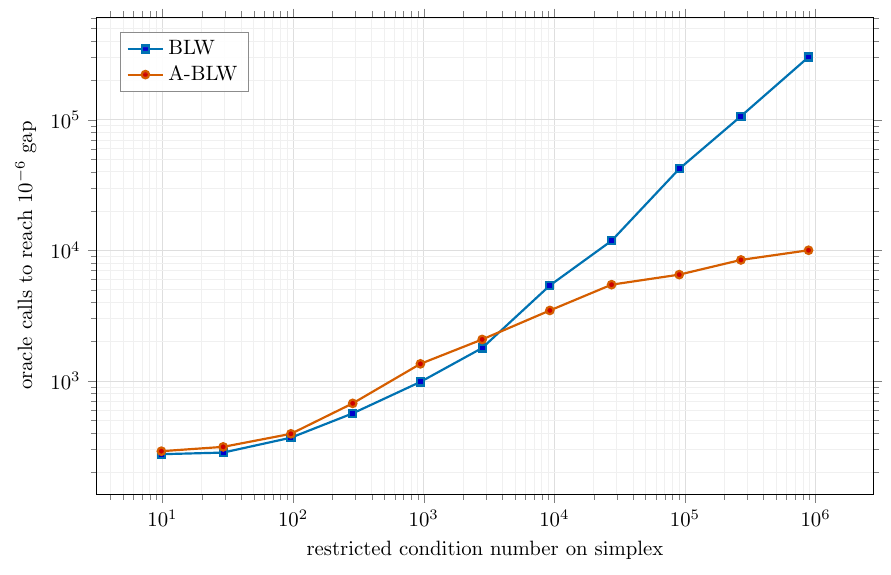}
	\caption{Simplex-constrained quadratic problems.  The plot shows the
		number of oracle calls needed to reach a best-point gap below \(10^{-6}\)
		as a function of the restricted condition number
		\(\kappa_\Delta=L_\Delta/\mu_\Delta\) on the simplex tangent space.}
	\label{fig:simplex-condition}
\end{figure}

The separation between the two methods becomes pronounced as the condition
number increases. For the largest tested instance, whose condition number is
\(8.82\cdot 10^5\), \(\mathrm{BLW}\) requires \(302{,}392\) oracle calls to
reach the target accuracy, whereas \(\mathrm{A\text{-}BLW}\) requires only
\(10{,}023\). This behavior is consistent with the theoretical complexity
bounds: for smooth problems satisfying quadratic growth, \(\mathrm{A\text{-}BLW}\)
has only a square-root dependence on the condition number, while
\(\mathrm{BLW}\) depends on it linearly.

\section{Conclusion}

We introduced the affine W-certificate as a mechanism for exploiting growth
without knowing its modulus. The certificate measures how slowly an affine
minorant can attain a prescribed amount of descent and therefore connects the
geometry of a bundle model with a rigorous optimality-gap bound. Embedding this
idea in the BLCI subroutine and the outer BLW search gives an anytime,
parameter-free method for nonsmooth convex objectives under quadratic growth,
with the optimal $O(M_0^2/(\mu\epsilon))$ leading oracle complexity and a
Fej\'er-monotone sequence of centers. Replacing BLCI by its accelerated
counterpart A-BLCI yields A-BLW, which adapts simultaneously to the unknown
H\"older smoothness regime and the unknown growth conditions. Finally, numerical experiments demonstrate the practical performance of the proposed methods.

\providecommand{\bls}{\textsc{BLS}}

\bibliographystyle{unsrt}
\bibliography{Reference}

@book{complexity,
	Author = {Nemirovsky, A.S. and Yudin, D.B.},
	Isbn = {0-471-10345-4},
	Mrclass = {90C25 (68C25)},
	Mrnumber = {702836},
	Note = {Translated from the Russian and with a preface by E. R. Dawson, Wiley-Interscience Series in Discrete Mathematics},
	Pages = {xv+388},
	Publisher = {John Wiley \& Sons, Inc., New York},
	Series = {A Wiley-Interscience Publication},
	Title = {Problem complexity and method efficiency in optimization},
	Year = {1983}}

@article{nesterov2015universal,
  title={Universal gradient methods for convex optimization problems},
  author={Nesterov, Yu},
  journal={Mathematical Programming},
  volume={152},
  number={1},
  pages={381--404},
  year={2015},
  publisher={Springer}
}

@inproceedings{karimi2016linear,
  title={Linear convergence of gradient and proximal-gradient methods under the polyak-{\l}ojasiewicz condition},
  author={Karimi, Hamed and Nutini, Julie and Schmidt, Mark},
  booktitle={Machine Learning and Knowledge Discovery in Databases: European Conference, ECML PKDD 2016, Riva del Garda, Italy, September 19-23, 2016, Proceedings, Part I 16},
  pages={795--811},
  year={2016},
  organization={Springer}
}

@article{lemarechal1995new,
  title={New variants of bundle methods},
  author={Lemar{\'e}chal, Claude and Nemirovskii, Arkadii and Nesterov, Yurii},
  journal={Mathematical programming},
  volume={69},
  number={1},
  pages={111--147},
  year={1995},
  publisher={Springer}
}

@article{han2023survey,
  title={Survey descent: A multipoint generalization of gradient descent for nonsmooth optimization},
  author={Han, XY and Lewis, Adrian S},
  journal={SIAM Journal on Optimization},
  volume={33},
  number={1},
  pages={36--62},
  year={2023},
  publisher={SIAM}
}

@book{nesterov2003introductory,
  title={Introductory lectures on convex optimization: A basic course},
  author={Nesterov, Yurii},
  volume={87},
  year={2003},
  publisher={Springer Science \& Business Media}
}

@book{lan2020first,
  title={First-order and stochastic optimization methods for machine learning},
  author={Lan, Guanghui},
  volume={1},
  year={2020},
  publisher={Springer}
}

@article{bubeck2015convex,
  title={Convex optimization: Algorithms and complexity},
  author={Bubeck, S{\'e}bastien and others},
  journal={Foundations and Trends{\textregistered} in Machine Learning},
  volume={8},
  number={3-4},
  pages={231--357},
  year={2015},
  publisher={Now Publishers, Inc.}
}

@misc{nesterov2025universal,
  title={Universal Complexity Bounds for Universal Gradient Methods in Nonlinear Optimization},
  author={Nesterov, Yurii},
  year={2025},
  eprint={2509.20902},
  archivePrefix={arXiv},
  primaryClass={math.OC},
  doi={10.48550/arXiv.2509.20902},
  url={https://arxiv.org/abs/2509.20902}
}

@article{li2025simple,
  title={A Simple Uniformly Optimal Method Without Line Search for Convex Optimization},
  author={Li, Tianjiao and Lan, Guanghui},
  journal={Mathematical Programming},
  year={2025},
  publisher={Springer},
  doi={10.1007/s10107-025-02250-z}
}

@misc{zhang2025linearlyconvergentalgorithmsnonsmooth,
  title={Linearly Convergent Algorithms for Nonsmooth Problems with Unknown Smooth Pieces},
  author={Zhang, Zhe and Sra, Suvrit},
  year={2025},
  eprint={2507.19465},
  archivePrefix={arXiv},
  primaryClass={math.OC},
  doi={10.48550/arXiv.2507.19465},
  url={https://arxiv.org/abs/2507.19465}
}

@article{lan2024optimalparameterfreegradientminimization,
  title={Optimal and Parameter-Free Gradient Minimization Methods for Convex and Nonconvex Optimization},
  author={Lan, Guanghui and Ouyang, Yuyuan and Zhang, Zhe},
  journal={Mathematical Programming},
  year={2026},
  publisher={Springer},
  doi={10.1007/s10107-026-02352-2}
}

@article{renegar2022restart,
  author    = {James Renegar and Benjamin Grimmer},
  title     = {A Simple Nearly-Optimal Restart Scheme For Speeding-Up First Order Methods},
  journal   = {Foundations of Computational Mathematics},
  volume    = {22},
  number    = {1},
  pages     = {211--256},
  year      = {2022},
  doi       = {10.1007/s10208-021-09502-2}
}

@article{bolte2017error,
  title={From error bounds to the complexity of first-order descent methods for convex functions},
  author={Bolte, J{\'e}r{\^o}me and Nguyen, Tien Son and Peypouquet, Juan and Suter, Bruce W},
  journal={Mathematical Programming},
  volume={165},
  number={2},
  pages={471--507},
  year={2017},
  publisher={Springer}
}

@article{lan2015bundle,
	title={Bundle-level type methods uniformly optimal for smooth and nonsmooth convex optimization},
	author={Lan, Guanghui},
	journal={Mathematical Programming},
	volume={149},
	number={1},
	pages={1--45},
	year={2015},
	publisher={Springer}
}

@article{doikov2025lower,
	title={Lower complexity bounds for minimizing regularized functions},
	author={Doikov, Nikita},
	journal={Optimization Letters},
	volume={19},
	number={9},
	pages={1759--1778},
	year={2025},
	doi={10.1007/s11590-025-02237-x}
}

@article{bolte2007lojasiewicz,
	title={The {\L}ojasiewicz inequality for nonsmooth subanalytic functions with applications to subgradient dynamical systems},
	author={Bolte, J{\'e}r{\^o}me and Daniilidis, Aris and Lewis, Adrian},
	journal={SIAM Journal on Optimization},
	volume={17},
	number={4},
	pages={1205--1223},
	year={2007},
	publisher={SIAM}
}

@article{zhou2026adabb,
	title={AdaBB: Adaptive Barzilai-Borwein method for convex optimization},
	author={Zhou, Danqing and Ma, Shiqian and Yang, Junfeng},
	journal={Mathematics of Operations Research},
	volume={51},
	number={1},
	pages={715--745},
	year={2026},
	publisher={INFORMS}
}

@article{lan2024auto,
	title={Auto-conditioned primal-dual hybrid gradient method and alternating direction method of multipliers},
	author={Lan, Guanghui and Li, Tianjiao},
	journal={arXiv preprint arXiv:2410.01979},
	year={2024}
}

@article{lan2024projected,
	title={Projected gradient methods for nonconvex and stochastic optimization: new complexities and auto-conditioned stepsizes},
	author={Lan, Guanghui and Li, Tianjiao and Xu, Yangyang},
	journal={arXiv preprint arXiv:2412.14291},
	year={2024}
}

@article{deng2026uniformly,
	title={Uniformly optimal and parameter-free first-order methods for convex and function-constrained optimization},
	author={Deng, Qi and Lan, Guanghui and Lin, Zhenwei},
	journal={INFORMS Journal on Computing},
	year={2026},
	publisher={INFORMS}
}

@article{yagishita2025simple,
	title={Simple linesearch-free first-order methods for nonconvex optimization},
	author={Yagishita, Shotaro and Ito, Masaru},
	journal={arXiv preprint arXiv:2509.14670},
	year={2025}
}

@article{ye2025simple,
	title={A Simple Adaptive Proximal Gradient Method for Nonconvex Optimization},
	author={Ye, Zilong and Ma, Shiqian and Yang, Junfeng and Zhou, Danqing},
	journal={arXiv preprint arXiv:2510.06079},
	year={2025}
}

@article{giang2026auto,
	title={Auto-Conditioned Frank-Wolfe Algorithms},
	author={Giang-Tran, Khanh-Hung and Shafiee, Soroosh and Ho-Nguyen, Nam},
	journal={arXiv preprint arXiv:2605.15512},
	year={2026}
}

@article{ji2026stochastic,
	title={Stochastic Auto-conditioned Fast Gradient Methods with Optimal Rates},
	author={Ji, Yao and Lan, Guanghui},
	journal={arXiv preprint arXiv:2604.06525},
	year={2026}
}

@article{wu2026universal,
	title={Universal and Parameter-free Gradient Sliding for Composite Optimization},
	author={Wu, Yan and Ouyang, Yuyuan and Zhang, Zhe and Luo, Qi},
	journal={arXiv preprint arXiv:2603.23492},
	year={2026}
}

@article{suh2025adaptive,
	title={An Adaptive and Parameter-Free Nesterov's Accelerated Gradient Method for Convex Optimization},
	author={Suh, Jaewook J and Ma, Shiqian},
	journal={arXiv preprint arXiv:2505.11670},
	year={2025}
}

@article{li2024problem,
	title={Problem-parameter-free decentralized nonconvex stochastic optimization},
	author={Li, Jiaxiang and Chen, Xuxing and Ma, Shiqian and Hong, Mingyi},
	journal={arXiv preprint arXiv:2402.08821},
	year={2024}
}

@article{borodich2025nesterov,
	title={Nesterov finds GRAAL: Optimal and adaptive gradient method for convex optimization},
	author={Borodich, Ekaterina and Kovalev, Dmitry},
	journal={arXiv preprint arXiv:2507.09823},
	year={2025}
}

@article{guigues2026universal,
title={Universal subgradient and proximal bundle methods for convex and strongly convex hybrid composite optimization},
author={Guigues, Vincent and Liang, Jiaming and Monteiro, Renato DC},
journal={Journal of Optimization Theory and Applications},
volume={208},
number={3},
pages={112},
year={2026},
publisher={Springer}
}

@article{rodomanov2024universal,
	title={Universal gradient methods for stochastic convex optimization},
	author={Rodomanov, Anton and Kavis, Ali and Wu, Yongtao and Antonakopoulos, Kimon and Cevher, Volkan},
	journal={arXiv preprint arXiv:2402.03210},
	year={2024}
}

@article{ou2025linesearch,
	title={Linesearch-free adaptive Bregman proximal gradient for convex minimization without relative smoothness},
	author={Ou, Hongjia and Latafat, Puya and Themelis, Andreas},
	journal={arXiv preprint arXiv:2508.01353},
	year={2025}
}

@article{sujanani2025efficient,
	title={Efficient parameter-free restarted accelerated gradient methods for convex and strongly convex optimization: A. sujanani, rdc monteiro},
	author={Sujanani, Arnesh and Monteiro, Renato DC},
	journal={Journal of Optimization Theory and Applications},
	volume={206},
	number={2},
	pages={52},
	year={2025},
	publisher={Springer}
}

@article{lin2015adaptive,
  title={An adaptive accelerated proximal gradient method and its homotopy continuation for sparse optimization},
  author={Lin, Qihang and Xiao, Lin},
  journal={Computational Optimization and Applications},
  volume={60},
  number={3},
  pages={633--674},
  year={2015},
  publisher={Springer},
  doi={10.1007/s10589-014-9694-4}
}

@article{fercoq2017restart,
  title={Adaptive restart of accelerated gradient methods under local quadratic growth condition},
  author={Fercoq, Olivier and Qu, Zheng},
  journal={IMA Journal of Numerical Analysis},
  volume={39},
  number={4},
  pages={2069--2095},
  year={2019},
  publisher={Oxford University Press},
  doi={10.1093/imanum/drz007}
}

@article{roulet2020sharpness,
  title={Sharpness, Restart, and Acceleration},
  author={Roulet, Vincent and d'Aspremont, Alexandre},
  journal={SIAM Journal on Optimization},
  volume={30},
  number={1},
  pages={262--289},
  year={2020},
  publisher={SIAM},
  doi={10.1137/18M1224568}
}

@article{wu2026restart,
  title={A Parameter-Free Restart Scheme with Only a Parallelizable {$\log\log(1/\epsilon)$} Overhead},
  author={Wu, Yue and Grimmer, Benjamin},
  journal={arXiv preprint arXiv:2605.30502},
  year={2026},
  doi={10.48550/arXiv.2605.30502}
}

@article{burke1993weak,
  title={Weak Sharp Minima in Mathematical Programming},
  author={Burke, James V. and Ferris, Michael C.},
  journal={SIAM Journal on Control and Optimization},
  volume={31},
  number={5},
  pages={1340--1359},
  year={1993},
  publisher={SIAM},
  doi={10.1137/0331063}
}

@article{drusvyatskiy2018error,
  title={Error Bounds, Quadratic Growth, and Linear Convergence of Proximal Methods},
  author={Drusvyatskiy, Dmitriy and Lewis, Adrian S.},
  journal={Mathematics of Operations Research},
  volume={43},
  number={3},
  pages={919--948},
  year={2018},
  publisher={INFORMS},
  doi={10.1287/moor.2017.0889}
}

@inproceedings{liu2017adaptive,
  title={Adaptive Accelerated Gradient Converging Methods under {H}{\"o}lderian Error Bound Condition},
  author={Liu, Mingrui and Yang, Tianbao},
  booktitle={Advances in Neural Information Processing Systems},
  volume={30},
  pages={3104--3114},
  year={2017}
}

@article{ito2021nearly,
  title={Nearly Optimal First-Order Methods for Convex Optimization under Gradient Norm Measure: An Adaptive Regularization Approach},
  author={Ito, Masaru and Fukuda, Mituhiro},
  journal={Journal of Optimization Theory and Applications},
  volume={188},
  number={3},
  pages={770--804},
  year={2021},
  publisher={Springer},
  doi={10.1007/s10957-020-01806-7}
}

@article{davis2018sharp,
  title={Subgradient Methods for Sharp Weakly Convex Functions},
  author={Davis, Damek and Drusvyatskiy, Dmitriy and MacPhee, Kellie J. and Paquette, Courtney},
  journal={Journal of Optimization Theory and Applications},
  volume={179},
  number={3},
  pages={962--982},
  year={2018},
  publisher={Springer},
  doi={10.1007/s10957-018-1372-8}
}

@article{davis2025local,
  title={A Local Nearly Linearly Convergent First-Order Method for Nonsmooth Functions with Quadratic Growth},
  author={Davis, Damek and Jiang, Liwei},
  journal={Foundations of Computational Mathematics},
  volume={25},
  number={3},
  pages={943--1024},
  year={2025},
  publisher={Springer},
  doi={10.1007/s10208-024-09653-y}
}

@article{zhang2025linearly,
  title={Linearly Convergent Algorithms for Nonsmooth Problems with Unknown Smooth Pieces},
  author={Zhang, Zhe and Sra, Suvrit},
  journal={arXiv preprint arXiv:2507.19465},
  year={2025},
  doi={10.48550/arXiv.2507.19465}
}

@article{lin2026optimal,
  title={Accelerated Prox-Level Methods for Unknown Piecewise-Smooth Optimization {I}: Convex Optimization},
  author={Lin, Zhenwei and Zhang, Zhe},
  journal={arXiv preprint arXiv:2601.14680},
  year={2026},
  doi={10.48550/arXiv.2601.14680}
}

@article{grimmer2023holder,
  title={General {H}{\"o}lder Smooth Convergence Rates Follow From Specialized Rates Assuming Growth Bounds},
  author={Grimmer, Benjamin},
  journal={arXiv preprint arXiv:2104.10196},
  year={2023},
  doi={10.48550/arXiv.2104.10196}
}

@article{malitsky2018golden,
	title={Golden ratio algorithms for variational inequalities},
	author={Malitsky, Yura},
	journal={arXiv preprint arXiv:1803.08832},
	year={2018}
}

@inproceedings{malitsky2020adaptive,
	title={Adaptive Gradient Descent without Descent},
	author={Malitsky, Yura and Mishchenko, Konstantin},
	booktitle={Proceedings of the 37th International Conference on Machine Learning (ICML)(2020)},
	volume={119},
	year={2020}
}

@article{marumo2024parameter,
	title={Parameter-free accelerated gradient descent for nonconvex minimization},
	author={Marumo, Naoki and Takeda, Akiko},
	journal={SIAM Journal on Optimization},
	volume={34},
	number={2},
	pages={2093--2120},
	year={2024},
	publisher={SIAM}
}

@article{drusvyatskiy2016generic,
	title={Generic minimizing behavior in semialgebraic optimization},
	author={Drusvyatskiy, Dmitriy and Ioffe, Alexander D and Lewis, Adrian S},
	journal={SIAM Journal on Optimization},
	volume={26},
	number={1},
	pages={513--534},
	year={2016},
	publisher={SIAM}
}

@article{davis2025gradient,
	title={Gradient descent with adaptive stepsize converges (nearly) linearly under fourth-order growth: D. Davis et al.},
	author={Davis, Damek and Drusvyatskiy, Dmitriy and Jiang, Liwei},
	journal={Mathematical Programming},
	pages={1--66},
	year={2025},
	publisher={Springer}
}

@article{charisopoulos2024superlinearly,
	title={A superlinearly convergent subgradient method for sharp semismooth problems},
	author={Charisopoulos, Vasileios and Davis, Damek},
	journal={Mathematics of Operations Research},
	volume={49},
	number={3},
	pages={1678--1709},
	year={2024},
	publisher={INFORMS}
}

@article{davis2024stochastic,
	title={Stochastic algorithms with geometric step decay converge linearly on sharp functions},
	author={Davis, Damek and Drusvyatskiy, Dmitriy and Charisopoulos, Vasileios},
	journal={Mathematical Programming},
	volume={207},
	number={1},
	pages={145--190},
	year={2024},
	publisher={Springer}
}

@article{kong2025lipschitz,
	title={Lipschitz minimization and the Goldstein modulus: S. Kong, AS Lewis},
	author={Kong, Siyu and Lewis, Adrian S},
	journal={Mathematical Programming},
	pages={1--30},
	year={2025},
	publisher={Springer}
}

@article{li2025subgradient,
	title={Subgradient Regularization: A Descent-Oriented Subgradient Method for Nonsmooth Optimization},
	author={Li, Hanyang and Cui, Ying},
	journal={arXiv preprint arXiv:2505.07143},
	year={2025}
}

@article{cox2014dual,
	author  = {Cox, Bruce and Juditsky, Anatoli and Nemirovski, Arkadi},
	title   = {Dual subgradient algorithms for large-scale nonsmooth learning problems},
	journal = {Mathematical Programming},
	year    = {2014},
	volume  = {148},
	number  = {1-2},
	pages   = {143--180},
	month   = dec,
	doi     = {10.1007/s10107-013-0725-1},
	url     = {https://doi.org/10.1007/s10107-013-0725-1}
}

\appendix

\section{Auxiliary results}\label{sec:appendix-auxiliary}

\subsection{A geometrically capped moment bound}
\label{sec:appendix-moment-bound}

\begin{lemma}[Geometrically capped moment bound]\label{lem:moment}
	Let $\rho\in[0,1]$, $N\ge0$, and $B>0$. Suppose that $a_0,\ldots,a_N\in\R_+$ satisfy
	\[
		\sum_{i=0}^N a_i^{\frac{1+3\rho}{1+\rho}}\le B,
		\qquad
		a_i^{\frac{1+3\rho}{1+\rho}}\le 2^{-i\frac{2\rho}{1+\rho}}B,
		\qquad i=0,\ldots,N.
	\]
	Then the explicit absolute bound
	\begin{equation}
		\sum_{i=0}^N a_i\le 5B^{\frac{1+\rho}{1+3\rho}}
	\end{equation}
	holds.
\end{lemma}

\begin{proof}
	If $\rho=0$, the assumption gives $\sum_{i=0}^N a_i\le B$, which is stronger than the claim. Hence assume $\rho\in(0,1]$. Let
	\[
		p:=\frac{1+3\rho}{1+\rho},\qquad \tau:=1-\frac1p=\frac{2\rho}{1+3\rho},\qquad b_i:=B^{-1/p}a_i.
	\]
	Then
	\[
		\sum_{i=0}^N b_i^p\le1,
		\qquad 0\le b_i\le 2^{-\tau i}.
	\]
	For any $k\in\{0,\ldots,N+1\}$, H\"older's inequality on the first $k$ terms and the geometric cap on the tail give
	\[
		\sum_{i=0}^N b_i\le \sum_{i=0}^{k-1}b_i+\sum_{i=k}^N2^{-\tau i}
		\le k^{1-1/p}+\frac{2^{-k\tau}\left(1-2^{-\tau(N-k+1)}\right)}{1-2^{-\tau}}.
	\]
	Since $1-1/p=\tau$, multiplying by $B^{1/p}$ gives the stated cutoff bound. To obtain the absolute bound, choose
	\[
		k=\min\left\{N+1,\left\lceil \frac{2}{\tau}\log_2\frac1\tau\right\rceil\right\}.
	\]
	If $k=N+1$, the tail is zero. Otherwise, $2^{-k\tau}\le \tau^2$, and hence
	\[
		\sum_{i=0}^N b_i\le \left\lceil \frac{2}{\tau}\log_2\frac1\tau\right\rceil^\tau+\frac{\tau^2}{1-2^{-\tau}}.
	\]
	We bound the two terms by numerical constants. Since $\tau\in(0,1/2]$ and $u=1/\tau\ge2$, the elementary inequality $\log_2 u\le u$ gives
	\[
		\left\lceil \frac{2}{\tau}\log_2\frac1\tau\right\rceil
		\le 1+\frac{2}{\tau^2}\le \frac{3}{\tau^2}.
	\]
	Therefore
	\[
		\left\lceil \frac{2}{\tau}\log_2\frac1\tau\right\rceil^\tau
		\le 3^\tau\tau^{-2\tau}\le \sqrt3 e^{2/e}<4,
	\]
	where we used $\sup_{s>0}s^{-2s}=e^{2/e}$. Also, the concavity of $1-2^{-s}$ on $[0,1/2]$ implies $1-2^{-\tau}\ge \tau/2$, and hence
	\[
		\frac{\tau^2}{1-2^{-\tau}}\le 2\tau\le1.
	\]
	Thus $\sum_{i=0}^N b_i\le5$, and multiplying by $B^{1/p}$ proves the claim.
\end{proof}

\subsection{Technical lemmas for extending the A-BLW method to other growth conditions}
\label{sec:appendix-technical-growth}

When $\rho=1$, each outer iteration of A-BLW costs on the order of
$(M_1/\mu_k)^{1/2}$ oracle evaluations. Therefore, removing the extra
multiplicative logarithm in the convex and higher-order-growth cases requires
a more refined lower bound on $\mu_k$ in terms of the progress measure
$\mu_k\Delta_k$. The following technical lemma is useful for both the convex
and the higher-order growth settings.

\begin{lemma}[Gap parameter controlled by a modulus-dependent gap bound]
	\label{lem:gap-parameter-controlled-by-gap}
	Consider the iterates $\{(\bar x_k,\mu_k,\Delta_k)\}_{k\ge0}$ generated by
	A-BLW, with $\mu_0$ chosen according to~\eqref{eq:mu0-def} and
	\[
		\Delta_0=\frac{2\|f'(\bar x_0)\|^2}{\mu_0}.
	\]
	Let $F:\R_+\to\R_+$ be any nonnegative function such that
	\begin{equation}\label{eq:F-gap-bound}
		f(\bar x_k)-f^*\le F(\mu_k),\qquad k\ge0.
	\end{equation}
	Then, for every $k\ge0$,
	\[
		\Delta_k\le \frac{5}{2}F(\mu_k).
	\]
\end{lemma}

\begin{proof}
	Let $\delta_k:=f(\bar x_k)-f^*$. We first verify the claim at $k=0$.
	Since $\bar x_0$ is chosen as whichever of the two points $\hat x_1$ and $\hat x_2$ has the larger objective value, we have
	\[
		|f(\hat x_1)-f(\hat x_2)|\le f(\bar x_0)-f^*=\delta_0.
	\]
	Moreover,
	\[
		\mu_0\ge
		\frac{\|f'(\bar x_0)\|^2}{|f(\hat x_1)-f(\hat x_2)|}.
	\]
	Hence
	\[
		\Delta_0=\frac{2\|f'(\bar x_0)\|^2}{\mu_0}
		\le 2\delta_0
		\le 2F(\mu_0)
		\le \frac{5}{2}F(\mu_0).
	\]

	We now prove the result by induction. Suppose that
	$\Delta_{k-1}\le \frac{5}{2}F(\mu_{k-1})$. There are two cases. First, suppose that
	$\mu_k=\mu_{k-1}$. Then the accepted trial uses the initial gap parameter
	$\widehat\Delta_k=\frac34\Delta_{k-1}$, and hence
	\[
		\Delta_k=\frac34\Delta_{k-1}
		\le \frac34\cdot \frac{5}{2}F(\mu_{k-1})
		\le \frac{5}{2}F(\mu_k).
	\]

	Second, suppose that $\mu_k<\mu_{k-1}$. By construction of the embedded
	search, the trial immediately before the accepted one used the same center
	$\bar x_k$ and gap parameter $\Delta_k/2$, and returned a false
	certificate flag. By Item~\ref{item:ablw-small-gap-cert} of
	Proposition~\ref{prop:ablw}, a false certificate flag is possible only if
	\[
		\delta_k>\frac{2}{5}\Delta_k.
	\]
	Consequently, $\Delta_k<\frac{5}{2}\delta_k\le\frac{5}{2}F(\mu_k)$, where the last inequality
	follows from~\eqref{eq:F-gap-bound}. This completes the induction.
\end{proof}

The next corollary specializes
Lemma~\ref{lem:gap-parameter-controlled-by-gap} to the purely convex case with
a Fej\'er-monotone center sequence.

\begin{corollary}[Convex case with Fej\'er-monotone centers]
	\label{cor:convex-mu-lower-bound}
	Let
	\[
		R_0:=\dist(\bar x_0,X^*)\in(0,+\infty).
	\]
	Then, for every outer iteration $k$,
	\begin{equation}\label{eq:convex-mu-lower-bound}
		\mu_k\ge \frac{\sqrt{\mu_k\Delta_k}}{\sqrt5\,R_0}.
	\end{equation}
\end{corollary}

\begin{proof}
	Let $\delta_k:=f(\bar x_k)-f^*$. If $\delta_k=0$, then the bound
	$\delta_k\le2\mu_kR_0^2$ is immediate. Otherwise, since the centers have
	nonincreasing objective values, $\delta_j\ge\delta_k$ for every
	$0\le j\le k$. Also, Fej\'er monotonicity gives $\dist(\bar x_j,X^*)\le R_0$ for all such $j$.
	Therefore the sequential modulus
	\[
		\mu_k^{\rm seq}:=
		\inf_{0\le j\le k}
		\frac{2\delta_j}{\dist^2(\bar x_j,X^*)}
	\]
	satisfies
	\begin{equation}\label{eq:convex-seq-lower-bound}
		\mu_k^{\rm seq}\ge \frac{2\delta_k}{R_0^2}.
	\end{equation}
	Applying the same induction argument used in
	Proposition~\ref{prop:outer-ablw} with the sequential modulus
	$\mu_k^{\rm seq}$ gives
	\begin{equation}\label{eq:mu-seq-quarter}
		\mu_k\ge \frac14\mu_k^{\rm seq}.
	\end{equation}
	Combining~\eqref{eq:convex-seq-lower-bound} and
	\eqref{eq:mu-seq-quarter} yields
	\[
		\delta_k\le 2\mu_kR_0^2.
	\]
	Thus Lemma~\ref{lem:gap-parameter-controlled-by-gap} applies with
	$F(\mu)=2\mu R_0^2$, and gives
	\[
		\Delta_k\le \frac{5}{2}F(\mu_k)=5\mu_kR_0^2.
	\]
	Multiplying both sides by $\mu_k$ gives
	$\mu_k\Delta_k\le5\mu_k^2R_0^2$, which is equivalent to
	\eqref{eq:convex-mu-lower-bound}.
\end{proof}

Finally, we give the corresponding estimate under higher-order growth.

\begin{corollary}[Higher-order growth]
	\label{cor:higher-order-mu-lower-bound}
	Suppose that $f$ satisfies the higher-order growth condition
	\begin{equation}\label{eq:higher-order-growth}
		f(x)-f^*\ge \mu_\alpha\dist^\alpha(x,X^*),\qquad x\in X_0,
	\end{equation}
	with $\alpha>2$. Then
	\begin{equation}\label{eq:higher-order-mu-lower-bound}
		\mu_k\ge
		c_\alpha\,\mu_\alpha^{1/(\alpha-1)}
		(\mu_k\Delta_k)^{\frac{\alpha-2}{2(\alpha-1)}},
		\qquad
		c_\alpha:=
		\left(\frac{5}{2}\right)^{-\frac{\alpha-2}{2(\alpha-1)}}
		2^{-\frac{\alpha}{2(\alpha-1)}}.
	\end{equation}
\end{corollary}

\begin{proof}
	Let $\delta_k:=f(\bar x_k)-f^*$.
	If $\delta_k>0$, then, since the center gaps are nonincreasing, the
	higher-order growth condition~\eqref{eq:higher-order-growth} gives the
	sequential modulus bound
	\begin{equation}\label{eq:mu-alpha-eps}
		\mu_k^{\rm seq}\ge
		2\mu_\alpha^{2/\alpha}\delta_k^{(\alpha-2)/\alpha}.
	\end{equation}
	The sequential-modulus version of Proposition~\ref{prop:outer-ablw} gives
	\[
		\mu_k\ge \frac14\mu_k^{\rm seq}
		\ge \frac12\mu_\alpha^{2/\alpha}\delta_k^{(\alpha-2)/\alpha}.
	\]
	Since $\alpha>2$, this rearranges to
	\[
		\delta_k
		\le
		\left(\frac{2\mu_k}{\mu_\alpha^{2/\alpha}}\right)^{\frac{\alpha}{\alpha-2}}.
	\]
	The same bound is trivial when $\delta_k=0$.
	Thus Lemma~\ref{lem:gap-parameter-controlled-by-gap} applies with
	\[
		F(\mu)=
		\left(\frac{2\mu}{\mu_\alpha^{2/\alpha}}\right)^{\frac{\alpha}{\alpha-2}},
	\]
	so
	\[
		\mu_k\Delta_k
		\le
		\frac{5}{2}\mu_k
		\left(\frac{2\mu_k}{\mu_\alpha^{2/\alpha}}\right)^{\frac{\alpha}{\alpha-2}}.
	\]
	Equivalently,
	\[
		\mu_k\Delta_k
		\le
		\frac{5}{2}\cdot2^{\frac{\alpha}{\alpha-2}}
		\mu_\alpha^{-\frac{2}{\alpha-2}}
		\mu_k^{\frac{2(\alpha-1)}{\alpha-2}}.
	\]
	Solving for $\mu_k$ gives
	\[
		\mu_k\ge
		\left(\frac{5}{2}\right)^{-\frac{\alpha-2}{2(\alpha-1)}}
		2^{-\frac{\alpha}{2(\alpha-1)}}
		\mu_\alpha^{1/(\alpha-1)}
		(\mu_k\Delta_k)^{\frac{\alpha-2}{2(\alpha-1)}},
	\]
	which is the claimed bound.
\end{proof}

\end{document}